\documentclass{amsart}

\usepackage[hmargin=2.5cm,bmargin=2.5cm,tmargin=3cm]{geometry}
\usepackage{soul}

\usepackage{amsmath,amssymb,amsthm,amsfonts,enumerate,url}
\usepackage{mathrsfs}
\usepackage{graphicx}

\usepackage{hyperref}

\usepackage{tikz}
\usetikzlibrary{automata,positioning,arrows.meta}
\usepackage{pgfplots}
\usepgfplotslibrary{polar}
\usepackage{float}

\usepackage{faktor}

\usepackage[utf8]{inputenc}
\usepackage{orcidlink}

\theoremstyle{plain}
\newtheorem{theo}{Theorem}[section]
\newtheorem{lemma}[theo]{Lemma}
\newtheorem{prop}[theo]{Proposition}
\newtheorem{cor}[theo]{Corollary}

\theoremstyle{definition}
\newtheorem{defi}[theo]{Definition}

\theoremstyle{remark}
\newtheorem{rem}[theo]{Remark}

\newtheorem{exa}[theo]{Example}

\numberwithin{equation}{section}
\numberwithin{figure}{section}

\newcommand\Graph{\mathcal{G}}

\newcommand{\ud}{\,\mathrm{d}}

\newcommand{\R}{\mathbb{R}}

\newcommand{\N}{\mathbb{N}}

\newcommand{\Dir}{\mathfrak{V}}

\usepackage{scalerel}
\usepackage[usestackEOL]{stackengine}

\def\dashint{\,\ThisStyle{\ensurestackMath{%
            \stackinset{c}{.2\LMpt}{c}{.5\LMpt}{\SavedStyle-}{\SavedStyle\phantom{\int}}}%
        \setbox0=\hbox{$\SavedStyle\int\,$}\kern-\wd0}\int}

 \def\mGraph{\mathcal{G}}

 \def\mV{\mathsf{V}}
 
 \def\mE{\mathsf{E}}

 \def\mv{\mathsf{v}}
 \def\me{\mathsf{e}}
 \def\mw{\mathsf{w}}
 
 \def\mf{\mathsf{f}}

\def\centervertex{\mathsf{v}_{\mathsf{c}}} 
\def\boundarymap{\beta_{\Graph}}

\DeclareMathOperator{\diam}{Diam}

\DeclareMathOperator{\dist}{dist}

\DeclareMathOperator{\dista}{dist}

\DeclareMathOperator{\essinf}{ess \, inf}

\DeclareMathOperator{\dom}{dom}
\DeclareMathOperator{\form}{\mathfrak{t}}
\DeclareMathOperator{\Inr}{\operatorname{Inr}}
\DeclareMathOperator{\Girth}{\operatorname{Girth}}

\usepackage[makeroom]{cancel}
\usepackage[normalem]{ulem}

\usepackage{verbatim}
\usepackage{array}

\title{
On the spectrum of infinite quantum graphs
} 

\author[M.~Düfel]{Marco Düfel}
\author[J.B.~Kennedy]{James B.~Kennedy\orcidlink{0000-0001-5634-0301}}
\author[D.~Mugnolo]{Delio Mugnolo\orcidlink{0000-0001-9405-0874}}
\author[M.~Plümer]{Marvin Plümer\orcidlink{0000-0002-7161-4697}}
\author[M.~Täufer]{Matthias Täufer\orcidlink{0000-0001-8473-2310}}

\address{Marco Düfel, Lehrgebiet Analysis, Fakultät Mathematik und Informatik, Fern\-Universität in Hagen, D-58084 Hagen, Germany}
\email{marcoduefel@gmail.com}

\address{James B.~Kennedy, Departamento de Matem\'atica, Universidade de Aveiro, P-3810-193 Aveiro, Portugal}
\email{jbkennedy@ua.pt}

\address{Delio Mugnolo, Lehrgebiet Analysis, Fakultät Mathematik und Informatik, Fern\-Universität in Hagen, D-58084 Hagen, Germany}
\email{delio.mugnolo@fernuni-hagen.de}

\address{Marvin Plümer, Lehrgebiet Analysis, Fakultät Mathematik und Informatik, Fern\-Universität in Hagen, D-58084 Hagen, Germany}
\email{marvin.pluemer@online.de}

\address{Matthias Täufer, Universit\'e Polytechnique Hauts-de-France, C\'ERAMATHS/DMATHS,
F-59313 -- Valenciennes Cedex 9 France}
\email{matthias.taufer@uphf.fr}



\thanks{The work of J.B.K. was supported by the Funda\c{c}\~ao para a Ci\^encia e a Tecnologia, Portugal, via projects 2023.13921.PEX, \url{https://doi.org/10.54499/2023.13921.PEX}, UIDB/00208/2020, \url{https://doi.org/10.54499/UIDB/00208/2020} and UIDB/04106/2020, \url{https://doi.org/10.54499/UIDB/04106/2020}. The work of D.M. and M.P.\ was supported by the Deutsche Forschungsgemeinschaft (Grant 397230547). This article is based upon work from COST Action 18232 MAT-DYN-NET, supported by COST (European Cooperation in Science and Technology), \url{www.cost.eu}.}

\begin{document}

\begin{abstract}
We study the interplay between spectrum, geometry and boundary conditions for two distinguished self-adjoint realisations of the Laplacian on infinite metric graphs, the so-called \emph{Friedrichs} and \emph{Neumann extensions}.
We introduce a new criterion for compactness of the resolvent and apply this to identify a transition from purely discrete to non-empty essential spectrum among a class of infinite metric graphs, a phenomenon that seems to have no known counterpart for Laplacians on Euclidean domains of infinite volume.
In the case of discrete spectrum we then prove upper and lower bounds on eigenvalues, thus extending a number of bounds previously only known in the compact setting to infinite graphs. 
Some of our bounds, for instance in terms of the inradius, are new even on compact graphs.

\end{abstract}

\maketitle


\section{Introduction}\label{sec:intro}

This article is about spectral properties of different self-adjoint realisations of the Laplace operator on infinite metric graphs. 
Recently, it was observed in~\cite{KosMugNic22} that on infinite metric graphs there can be points in their closure which do not belong to the graph itself (so-called \emph{graph ends}),
 and that two natural realisations of the Laplacian, namely the \emph{Friedrichs} and the \emph{Neumann extension}, can be understood to impose Dirichlet and Neumann conditions, respectively, at those ends.
This raises the question of what effect the choice of the extensions will have on the spectrum.
The answer is, roughly speaking, that the choice of the extension is quite relevant and can have more dramatic consequences on infinite quantum graphs than the choice of vertex conditions has on Laplacians on compact graphs.

We develop qualitative criteria for the absence or existence of essential spectrum, encountering surprises such as a phase transition in the spectrum of the Neumann extension. In the case of purely discrete spectrum, we also prove quantitative estimates on eigenvalues, some of which extend estimates previously only known for finite quantum graphs, and some of which are entirely new.

By quantum graphs we understand Schrödinger operators acting on functions supported on metric graphs.
In the case of \emph{compact metric graphs}, that is, metric graphs consisting of a finite set of edges, each of finite length, they were introduced to the mathematical literature in the early 1980s~\cite{Lum80,Lum80b}. 
It was proved in~\cite{PavFad83} that, upon imposing appropriate transmission conditions in the vertices, they are self-adjoint and have purely discrete spectrum.
Since then, describing spectral properties of quantum graphs in terms of geometric properties of the underlying metric graph, and vice versa, has become an increasingly popular topic, especially in the last decade, see~\cite{BerKenKur19,Kur23} and references therein.

Infinite quantum graphs, and in particular their spectral properties, have been studied since the late 1990s  in \cite{Car97,Cat97,BelLub04,BruGeyPan08} and more recently among others in~\cite{LenPan16,BecGreMug21}.
Therein, the relevant operator is essentially self-adjoint. The case where different self-adjoint extensions may exist was first discussed in the pioneering papers \cite{Car00,NaiSol00}, and then more extensively in \cite{BenSalSal11,ExnKosMal18,KosNic19}. In these papers, the spectral properties of only the Friedrichs extension of the (general realisations of) Laplacians on infinite quantum graphs were discussed explicitly.
First investigations relied on a transference principle relating the spectrum of a quantum graph to the spectrum of an infinite matrix: this idea goes back to von Below's early analysis in~\cite{Bel85} and relies fundamentally upon the assumption that \emph{all edges have the same length}. 
If this assumption is dropped, infinite quantum graphs will generally \emph{not} be essentially self-adjoint: the operator domain may then have to include appropriate \emph{boundary conditions at infinity}, which in turn will influence the spectrum.
A similar phenomenon is known for discrete Laplacians on combinatorial graphs, see e.g.~\cite{HuaKelMas13,GeoHaeKel15}. Also, let us point out that it has been known since \cite{AloKelTep16}, and implicitly since \cite{AloFre12}, that certain classes of fractals, including so-called Hanoi attractors, can be \textit{approximated by metric graphs}: in other words, Laplacians thereon are infinite quantum graphs, in our language. Spectral analysis for such classes of fractals has been performed since \cite{AloFre17,AloFreKig18,AloCheGu19}. However, we do note that important classes of infinite quantum graphs, even as natural as the Neumann realisation of the Laplacians on (fractal-like, but possibly non-selfsimilar) radially symmetric trees with infinitely many ends, are not covered by the theory in the present paper; and yet, it is known that under suitable growth conditions, they may have pure point spectrum: we refer to~\cite{Sol04,NicSem18,JolKacSem19} and especially to the recent paper \cite{FraNadPan25}.

As intimated above, the notion of infinite quantum graph we are going to discuss relies on the notion of \emph{ends}, a classical concept from graph theory~(see, e.g., \cite[Chapter 8]{Die17} and references therein) which was recently adapted to quantum graphs in~\cite{KosMugNic22}. 
As shown therein, if an end is ``thin enough'' in a certain sense, then different boundary conditions can be imposed on it leading to different self-adjoint realisations of the Laplacian. We will treat what are arguably the most important two types of conditions, Dirichlet and Neumann, in this paper.

Let us also point out parallels of this paper to the spectral geometry on domains and manifolds -- a well developed field, see, e.g.~\cite{GreNgu13} for a survey. 
If $\Omega\subset \R^d$ is a bounded domain with Lipschitz boundary, then its Dirichlet and Neumann Laplacians both have purely discrete spectrum, whereas on unbounded or non-smooth domains, the situation is more subtle, at least in the Neumann case, see, e.g., \cite[Chapter 6]{AdaFou03} or \cite{DavSim92}. In particular, a family of bounded planar domains (the \textit{combs}) was introduced in \cite{HemSecSim91} that displays a remarkable behaviour: the Neumann Laplacian thereon has empty absolutely continuous spectrum but its essential spectrum is non-empty and contains infinitely many dilated copies of the set of accumulation points of some (possibly unbounded) sequence, see~\cite[Theorem~3.4]{HemSecSim91}.

In the case of {infinite} metric graphs, the issue of discreteness of the spectrum is significantly less well understood: {examples are known that show that the spectrum may or may not be purely discrete}~\cite{Sol04,Hae14,KosNic19,GerTru21}.  
Remarkably, even  infinite equilateral graphs may have purely discrete spectrum: this can, e.g., be seen by applying a well-known transference principle~\cite{Cat97} to suitable infinite combinatorial graphs with purely discrete spectrum, like the infamous \textit{lamplighter graph} (all of whose eigenvalues have infinite multiplicity).
There already exist some abstract criteria for the compactness of embeddings of Sobolev spaces over metric measure spaces, as introduced in~\cite{Haj96,Che99} see, e.g.,~\cite[Theorem~4.59]{Che99} and~\cite[Theorem~2]{Kal99}.
In particular, it has been known  since~\cite{Car00} that on metric graphs the spectrum is discrete provided the graph has finite volume.
Other criteria of a more geometric flavour have been provided in~\cite[Corollary~3.5 and Corollary~4.5]{KosNic19} and \cite[Theorem~4.18]{ExnKosMal18}.
We also mention \cite{Sol04}, which provides criteria for the discreteness of the spectrum and investigates the asymptotics of the eigenvalues on so-called \emph{regular trees}.

This paper is organised as follows:
Section~\ref{sec:2} contains definitions of metric graphs, Sobolev spaces, and extensions of the Laplace operator.
In Section~\ref{sec:embedding}, we provide criteria on discreteness or non-discreteness of the spectrum of these Laplacian extensions: our main contribution in this context is Theorem~\ref{theo:duefel}, a new Kolmogorov--Riesz-type criterion that, in particular, gives a sufficient and necessary condition for compactness of the embedding of Sobolev spaces over metric graphs $\Graph$ into the canonical $L^2$-space in terms of ``thinness'' of the periphery of the graph, measured by the $L^2$-norm of functions supported outside compact subgraphs of $\Graph$. 
Another criterion for discreteness of the Friedrichs extension on infinite trees from~\cite{KosNic19} is shown not to generalise to the Neumann extension. 
Subsection~\ref{H10yesH1not} puts these criteria to use by investigating a parameter-dependent family $\Graph_\alpha$ of metric graphs (the \emph{diagonal combs}) on which the spectra of the Neumann extensions exhibit a noteworthy phase transition from purely discrete to non-empty essential spectrum, despite all having infinite volume: a phenomenon which somewhat recalls Neumann Laplacians on (finite volume) cusp domains but which has no parallels on domains of infinite volume.  
To the best of our knowledge, Theorem~\ref{theo:duefel} is the first criterion able to identify this transition in the context of quantum graphs. Beyond quantum graphs, similar phenomena are elusive and have been discovered only recently: we mention e.g.\ the celebrated results about phase transition from pure point spectrum to absolute continuous spectrum for almost Mathieu operators, see~\cite{JitLiu18} and references therein. Weaker but comparable results were recently obtained in the context of metric graphs in \cite{KosNic21,KenMugTau24}.

Having determined criteria for discreteness of the spectrum, the next natural question is to investigate which bounds on eigenvalues from finite metric graphs remain valid on infinite metric graphs.
Section~\ref{sec:symmetrisation} is about lower bounds on eigenvalues.
Subsection~\ref{subsec:symmetrisation} introduces a symmetrisation technique and uses it to prove isoperimetric inequalities on infinite quantum graphs.
Subsection~\ref{subsec:further_lower_bounds} complements this with further lower bounds in terms of diameter and inradius.

Finally, Section~\ref{sec:other} complements the lower bounds in Section~\ref{sec:symmetrisation} by upper bounds in terms of diameter, volume and the first Betti number.

We mention in particular Theorem~\ref{prop:Makai} and Theorem~\ref{TheoremBettiD}. 

The former is a quantum graph analogue of the Hersch--Makai inequality which yields a lower bound on the lowest eigenvalue of the Dirichlet Laplacian in terms of the inradius on simply connected domains or convex domains. 
The estimate is valid if and only if the metric graph is a tree with a so-called \emph{centre vertex}, cf. Definition\ref{def:centre_vertex}. 
This is a class of trees which does not necessarily require radial symmetry and which, in light of Theorem~\ref{prop:Makai} exhibits a striking analogy to convex planar domains, see Section~\ref{subsubsec:Inradius} for a discussion.
The other bound features the diameter and the first Betti number of the graph. 
Both results are also new for compact graphs.

\section{Infinite metric and quantum graphs}
	\label{sec:2}

\subsection{Notation and elementary properties}
	\label{subsec:notation}

We will work with unoriented metric graphs with finite or countably infinite sets of vertices $ \mV $ and edges $ \mE $. 
We follow the formalism in~\cite{Mug19}, referring to that article for further details; in particular, we associate with each $\me\in\mE$ a length $\ell_\me\in (0,\infty)$ and regard $\Graph$ as a triple $\mathcal G (\mV,\mathcal E,\equiv)$, where $\mV$ is the vertex set, $\mathcal E:=\bigsqcup\limits_{\me\in\mE} [0,\ell_\me]$ is the set of metric edges, and $\equiv$ is an equivalence relation on $\mathcal E$ that encodes the connectivity of $\Graph$. We explicitly allow \emph{loops} (edges incident to only one vertex) and \emph{parallel edges}.

\begin{defi}\label{DefinitionInducedSubgraph}
A \emph{subgraph} of $\mathcal G= (\mV,\mathcal E,\equiv)$ is a metric graph $\Graph'=\mathcal G' (\mV',\mathcal E',\equiv')$ with $\mV'\subset \mV$, $\mathcal E'\subset \mathcal E$ and $\equiv'\subset \equiv$.

Any such subgraph is called an \emph{induced subgraph} of $\Graph$ if for any vertices $\mv,\mw$ in $\Graph'$, the set of edges $\me_{\mv,\mw}$ incident in both of them in $\Graph'$ is equal to the corresponding set in $\Graph$. 
In this case we denote by $\partial\mGraph'$ the set of points in $\mV'$ which form the topological boundary of $\mGraph'$ as a subset of $\mGraph$.
\end{defi}

We denote by $\mE_\mv = \{ \me \in \mE: \me = \me_{\mv,\mw} \text{ for some } \mw \in \mV\}$ the set of all edges incident to $\mv$. 
The \emph{degree} of a vertex $\mv \in \mV$, $\deg (\mv)$, is the number of edges incident to it, where loops are counted twice. 
A vertex $\mv$ is called a \emph{dummy vertex} if $\deg (\mv) = 2$. 
The deletion (or insertion) of a finite number of dummy vertices, where two incident edges are replaced by one of the same total length (or vice versa), does not affect the topology or metric structure of the graph nor does it affect any of the Laplace-type operators defined on it we are going to consider in this article.

We say that $\Graph$ is \emph{locally finite} if $\deg(\mv) < \infty$ for all $\mv \in \mV$, and \emph{compact} if, additionally, its edge set $\mE$ is finite.

We emphasise that neither changing the connectivity of a metric graph on a given set of metric edges $\mE\simeq \bigsqcup\limits_{\me\in\mE}[0,\ell_\me]$, nor inserting or deleting dummy vertices changes the \emph{volume} of $\Graph$ 
\[
|\mGraph|:=\|\ell_\me\|_1
\]

In particular, while our definition requires all edges to have finite length, this is no restriction, since edges of infinite length can be broken up into a countably infinite set of edges of finite length by inserting a countable set of evenly spaced dummy vertices, and are thus covered by our framework.

A metric graph is indeed a metric space with respect to the 
canonical shortest-path metric induced by the one-dimensional Euclidean distance and by the combinatorial distance on the underlying discrete graph, see~\cite{Mug19} for details.

In this paper we exclusively work with connected metric graphs. 
The \emph{diameter} of $\Graph$ is
\begin{equation}
	\label{eq:Diameter}
	\diam(\Graph) 
	:= 
	\sup_{x,y\in \Graph}\dista_{\Graph}(x,y),
\end{equation}
and we note that $\diam(\mGraph)\le |\mGraph|$.
Finally, a simple argument shows that a locally finite(!) connected metric graph $\Graph$ is a compact graph if and only if it is compact as a metric space.

A metric graph is also, in a canonical way, a measure space, endowed with the direct sum of Lebesgue measures on each interval $(0,\ell_\me)$. 
Throughout the article we will assume that $ \mGraph $ is locally finite and connected, which implies in particular that $\mGraph$ is a \emph{metric measure space}, cf.~\cite[Section~3]{Stu06}, in the sense that each $x\in\Graph$ has a neighbourhood of finite measure.

We will be mostly interested in \emph{infinite} metric graphs, i.e., for which $\mE$ is a (countably) infinite set. 
Note that under our assumptions, this is equivalent to $\mV$ being infinite.
At the risk of being redundant, we recapitulate some notions that are elementary in the case of finite metric graphs, but may lead to ambiguities in the infinite case.

\begin{defi}\label{DefinitionBetti}
Let $\mGraph$ be a locally finite, connected metric graph.
\begin{enumerate}
	\item
	A \emph{walk} is the image of a continuous map $c \colon [0,1] \to \mGraph$.
	\item
	A \emph{path} is an injective walk.
	\item 
	A \emph{cycle} $C \subset \mGraph$ is a compact subset of $\mGraph$ such that, for all $x,y \in C$, $C \setminus \{x,y\}$ consists of precisely two disjoint paths in $\mGraph$ connecting $x$ and $y$.
	\item 
	The \emph{(first) Betti number} $\beta \in \N_0 \cup \{\infty\}$ of $\mGraph$ is the cardinality of any basis of the cycles of $\mGraph$, that is, the cardinality of any minimal set of cycles of $\mGraph$ whose union, treated as a subset of $\mGraph$, contains all cycles in $\mGraph$.
	\item We call $\Graph$ a \emph{tree} if it contains no cycles as subgraphs.
	\item 
	We call $\mGraph$ \emph{doubly (path) connected} if, for all $x,y \in \mGraph$, there exist two paths $P_1, P_2 \subset \mGraph$ connecting $x$ and $y$, such that $P_1$ and $P_2$ intersect at at most finitely many vertices.
\end{enumerate}
\end{defi}

\begin{rem}\label{RemarkAlmostTree}
If $\beta < \infty$, then $\mGraph$ in essence consists of a compact ``core'', i.e., subgraph, to which finitely many tree subgraphs are attached. 
More precisely, since the union of all its (necessarily only finitely many) cycles will be compact, there exists a (generally non-unique) compact, connected subgraph $\mathcal{K}$ of $\mGraph$ which contains them all; thus, $\mGraph \setminus \mathcal{K}$ is a disjoint union of finitely many trees and each attached to $\mathcal{K}$ at a single vertex. These trees may be finite or infinite, and may have finite or infinite diameter. But we see that $\beta = 0$ if and only if $\mGraph$ is itself a tree. 
\end{rem}

Whenever $\mE$ is infinite, the \emph{ends} of $\Graph$ can be defined as in~\cite[Definition~3.1]{KosMugNic22} based on a classical combinatorial approach. For our purposes, a different (but equivalent, see~\cite[Section~2.2]{KosMugNic22} and references therein) notion of end of topological flavour is more convenient.

\begin{defi}[Topological end]\label{Topologisches Ende}
Consider a sequence $ \mathcal{U}=(U_n) $
of non-empty open connected subsets of $ \mGraph $ with compact boundaries and such that $ U_{n+1} \subset U_{n} $ for all $ n \in \mathbb{N} $ and $ \bigcap_{n \in \mathbb{N}} \overline{U_n} = \emptyset $. Two such sequences $ \mathcal{U}=(U_n)$ and $  \mathcal{U'}=(U'_n)$ are called \emph{equivalent} if for all $ n \in \mathbb{N} $ there exist $ j,k $, such that  $U'_j \subset U_n $ and $ U_k \subset U'_n $. An equivalence class  $ \gamma $ of sequences is called a \emph{topological end} -- or simply an \emph{end} -- of  $ \mGraph $ and $ \mathfrak{C}(\mGraph) $ denotes the set of all (topological) ends of $ \mGraph $.  \end{defi}

\begin{rem}\label{rem:redu}
Let us stress that \emph{ends are not vertices}. In particular, an end cannot be an endpoint of an edge. Thus a subset or subgraph of $\Graph$ is compact if and only if it is closed in $\Graph$ and intersects a finite number of edges of $\Graph$.
\end{rem}

Denote by $\overline{\mGraph}$ the metric completion of $\mGraph$.
Note that if $|\mGraph| < \infty$, then $\overline{\mGraph}$ is a compact metric space.

\begin{defi} \label{def:finvol}
An end $\gamma \in \mathfrak{C}(\mGraph)$ has \emph{finite volume} if there is a sequence $\mathcal U = (U_n)$  representing $\gamma$ such that the volume $|U_n|$ is finite for some $n$.
Otherwise, we say that $\gamma$ has \emph{infinite volume}.
\end{defi}

Here and throughout we will denote by $ \mathfrak{C}(\mGraph)$ and $ \mathfrak{C}_0(\mGraph)$ the set of ends of $\Graph$ and the set of ends of finite volume of $\Graph$, respectively. Note that for an end of finite volume we do not have a canonical way to define its "volume" as a non-negative number; we merely assert that the end is contained in a subgraph of $\mGraph$ of finite volume. 
In particular, $\mGraph$ will itself have finite volume if and only if all of its ends are of finite volume.

A related concept to the \emph{metric completion} $\overline \mGraph$ is the \emph{Freudenthal compactification} of $\mGraph$, the coarsest extension of the canonical topology of $\mGraph$ to $\mGraph \cup \mathfrak{C}(\mGraph)$ which renders this space compact.	
The relation between these two spaces is a delicate business. 
We refer to Appendix~\ref{app:completions} for a proper discussion and only summarise the following important facts:
There is a canonical way to identify points in $\overline \mGraph \backslash \mGraph$ with ends $\gamma \in \mathfrak{C}(\mGraph)$, see Proposition~\ref{prop:identifiying-metric-with-top-ends}.
However this map neither needs to be injective nor surjective, see Example~\ref{exa:marvin-ladder}.
Nevertheless, in important special cases, this map will lead to a homeomorphism which ensures that under certain conditions the notions of metric completion and the Freudenthal compactification coincide.
This includes graphs where all ends have finite volume, cf.\ Corollary~\ref{cor:all_graphs_end_finite_volume} (hence in particular graphs of finite volume as in Sections~\ref{sec:symmetrisation} and~\ref{sec:other}), as well as graphs of finite diameter with only one end such as the diagonal combs investigated in Section~\ref{H10yesH1not}.

\subsection{Compact exhaustions}
	\label{subsec:compact}
For an infinite graph $\mGraph$, we are going to consider appropriate sequences of compact subgraphs $\mGraph_n$ which approximate $\mGraph$ in a suitable way (cf.\ \cite[Section~2.2]{KosMugNic22}).
This will allow us to generalise many properties of $\mGraph_n$ directly to $\mGraph$, in particular some eigenvalue bounds.

\begin{defi}\label{DefinitionApprox}
Let $\mGraph$ be a locally finite, connected metric graph. 
A sequence of induced subgraphs $(\mGraph_n)_{n\in\N}$ is called a \emph{compact exhaustion} of $\mGraph$ if the following conditions are satisfied:
\begin{enumerate}
\item $\mGraph_1 \subset \mGraph_2 \subset \ldots \subset \mGraph$;
\item for all $\me \in \mE$ there exists $n \in \N$ such that $\me \subset \mGraph_n$;
\item each $\mGraph_n$ is connected;
\item for all $n \in \N$, $\mGraph_n$ is a compact subgraph of $\mGraph$.
\end{enumerate}
\end{defi}

Recall from Remark~\ref{rem:redu} that compact subgraphs of $\mGraph$ are exactly those which intersect only a finite number of edges of $\mGraph$. 
Properties (1) and (2) imply in particular $\bigcup_{n \in \N} \mGraph_n = \mGraph$ and that, if $\mGraph$ is compact, then every compact exhaustion of $\mGraph$ will eventually become stationary.

\begin{exa}\label{ExampleApprox}
We will often use the following construction, which yields a compact exhaustion in any locally finite, connected graph $\mGraph$: Fix a vertex $\mv \in \mV$.  Then, any $\mw \in \mV$ has a well-defined combinatorial distance to $\mv$, see for instance~\cite[Section~1.3]{Die17}.
Set
\[
	\mGraph_{\mv,n}
\] to be the induced subgraph of $\mGraph$ containing all vertices of $\mGraph$ of combinatorial distance at most $n \in \N$ to $\mv$. 
Recall from Definition~\ref{DefinitionInducedSubgraph} that the adjacency relations in $\mGraph_{\mv,n}$ are chosen to mirror those in $\mGraph$.
Then by construction $(\mGraph_{\mv,n})_{n\in\N}$ is a compact exhaustion of $\mGraph$. 
In fact, more is true: if $\mGraph$ has an infinite edge set, then $\mGraph_{\mv,n}$ is \emph{compactly} contained in a proper open subset of $\mGraph_{\mv,n+1}$ for all $n \in \N$.
\end{exa}

The following Lemma~\ref{LemmaEquivApprox}, which states that all compact exhaustions are equivalent, follows directly from Definition~\ref{DefinitionApprox}.
In particular, it justifies that we always use the compact exhaustions of Example~\ref{ExampleApprox}. 

\begin{lemma}\label{LemmaEquivApprox}
Let $\mGraph$ be a locally finite, connected metric graph, let $\mv \in \mV$ be any vertex of $\mGraph$, let $(\mGraph_{\mv,n})_{n\in\N}$ be the compact exhaustion from Example~\ref{ExampleApprox}, and let $(\mathcal{G}_n)_{n\in\N}$ be any other compact exhaustion of $\mGraph$. Then for all sufficiently large $n \in \N$ there exist positive integers $k_1 = k_1 (n)$ and $k_2 = k_2 (n)$ such that
\begin{equation}
\label{eq:EquivApprox}
	\mGraph_{\mv,k_1} \subset \mathcal{G}_n \subset \mGraph_{\mv,k_2}.
\end{equation}
Moreover, $k_2 \to \infty$, and $k_1$ can be chosen to tend to $\infty$, as $n \to \infty$.
\end{lemma}

For compact $\mGraph$ there are two equivalent ways to define the first Betti number $\beta$: Either as the number of independent cycles in $\mGraph$, as in Definition~\ref{DefinitionBetti}, or as $\beta = \# \mE - \# \mV + 1$. 
Obviously, the second definition no longer has a meaning on infinite graphs. 
However, it can be obtained via compact exhaustions, as stated in the following proposition.

\begin{prop}\label{PropBetti}
Let $\mGraph$ be a locally finite, connected metric graph with Betti number $\beta = \beta (\mGraph)$, and let $(\Graph_n)_{n \in \N}$ be a compact exhaustion of $\mGraph$. 
\begin{enumerate}
\item If $\beta < \infty$, then $\beta (\mGraph_n)$ is eventually constant and equal to $\beta$.
\item If $\beta = \infty$, then $\beta (\mGraph_n) \to \infty$.
\end{enumerate}
\end{prop}

In particular, for infinite graphs, $\beta$ can alternatively be defined as
\begin{displaymath}
	\lim_{n \to \infty} \beta (\mGraph_n) = \lim_{n\to \infty} \left( \# \mE (\mGraph_n) - \# \mV (\mGraph_n) + 1 \right),
\end{displaymath}
independently of the choice of the compact exhaustion $(\mGraph_n)_{n\in\N}$.

\begin{proof}
Without loss of generality, we may assume that $\mGraph$ is infinite. 
Note that (1) $\beta (\mGraph_n) \in \N_0$ for all $n \in \N$; and (2) every cycle in $\mGraph_n$ is also a cycle in $\mGraph$, so $\beta (\mGraph_n) \leq \beta$ for all $n \in \N$. 
Fix $\mv \in \mV$ arbitrary and let $(\mGraph_{\mv,n})_{n\in\N}$ be the compact exhaustion from Example~\ref{ExampleApprox}. 
Then, \eqref{eq:EquivApprox} and the fact that the $\mGraph_{\mv,n}$ as induced subgraphs have maximal connectivity imply that (with $k_1$ and $k_2$ as in \eqref{eq:EquivApprox})
\begin{equation}\label{EquationApproximationBetti}
	\beta (\mGraph_{\mv,k_1}) \leq \beta (\mGraph_n) \leq \beta (\mGraph_{\mv,k_2}),
\end{equation}
for sufficiently large $n$.

Thus it suffices to prove the proposition for the sequence $(\mGraph_{\mv,n})_{n\in\N}$. 
Note that the maximal connectivity also implies that $\beta(\mGraph_{\mv,n})$ is a non-decreasing sequence in $n \in \N$.
If $\beta < \infty$, there is a finite basis of cycles of $\mGraph$. 
Since every point on every cycle in this basis may be reached from $\mv$ via a path consisting of a finite number of edges, all of these cycles will eventually be contained in $\mGraph_{\mv,n}$.
In particular, $\beta (\mGraph_{\mv,n}) = \beta$ for all sufficiently large $n \in \N$.

If $\beta = \infty$, then for every $m \in \N$ one finds $m$ independent cycles which contain finite sets of edges and vertices. 
Thus, for sufficiently large $n$, they will be completely contained in $\mGraph_{\mv,n}$, and so $\beta(\mGraph_{\mv,n}) \geq m$.
Since $m$ was arbitrary, this shows the claim.
\end{proof}

We turn to the diameter of compact exhaustions. 
The assumption that they are connected ensures that they all have finite diameter.

\begin{prop}\label{PropDiameterBelow}
Let $\mGraph$ be a locally finite, connected metric graph and let $(\mGraph_n)_{n \in \N}$ be any compact exhaustion of $\mGraph$. 
If $\mGraph$ has finite diameter $\diam(\mGraph)$, then
	\begin{equation}
	\label{eq:InequalityInfiniteDiameter}
	\diam(\mGraph) 
	\leq 
	\liminf_{n\to\infty} \diam(\mGraph_n).
\end{equation}
If the diameter of $\mGraph$ is infinite, then
\begin{equation}\label{EquationInfiniteDiameter}
	\lim_{n\to\infty} \diam(\mGraph_n) 
	=
	\diam(\mGraph) = \infty.
\end{equation} 
\end{prop}

\begin{proof}
Let $x_k,y_k \in \mGraph$ be such that $\dista_{\mGraph} (x_k,y_k) \to \diam(\mGraph) \in [0,\infty]$. Since $(\mGraph_n)_{n \in \N}$ is a compact exhaustion, by properties (1) and (2), for each $k \in \N$ there exists some $n_k \in \N$ such that $x_k,y_k \in \mGraph_n$ for all $n \geq n_k$. Now since $\mGraph_n \subset \mGraph$ we clearly have $\dista_{\mGraph_n} (x,y) \geq \dista_{\mGraph} (x,y)$ for all $x,y \in \mGraph_n$ and all $n \in \N$; putting all this together, it follows that
\begin{displaymath}
	\dista_{\mGraph} (x_k,y_k) \leq \dista_{\mGraph_n} (x_k,y_k) \leq \diam(\mGraph_n)
\end{displaymath}
for all $n \geq n_k$. Since $\dista_{\mGraph} (x_k,y_k) \to \diam(\mGraph)$, the claim follows.
\end{proof}

\begin{exa}
We consider the infinite ladder graph \(\Graph\) as depicted in Figure \ref{fig:ladder-graph}. The graph edges will be denoted by \(\me_n, \mf_n^1\) and \(\mf_n^2\) for \(n\in\mathbb N\) (see Figure \ref{fig:ladder-graph} for the notation).
\begin{figure}[ht]
\centering
\begin{tikzpicture}
\coordinate (a1) at (0,0);
\coordinate (b1) at (2,0);
\coordinate (c1) at (4,0);
\coordinate (d1) at (6,0);
\coordinate (e1) at (8,0);
\coordinate (f1) at (9,0);
\coordinate (a2) at (0,2);
\coordinate (b2) at (2,2);
\coordinate (c2) at (4,2);
\coordinate (d2) at (6,2);
\coordinate (e2) at (8,2);
\coordinate (f2) at (9,2);
\draw[thick]  (a1) edge node [below] {\(\mf_1^1\)} (b1);
\draw[thick]  (b1) edge node [below] {\(\mf_2^1\)} (c1);
\draw[thick]  (c1) edge node [below] {\(\mf_3^1\)} (d1);
\draw[thick]  (d1) edge node [below] {\(\mf_4^1\)} (e1);
\draw[thick]  (a2) edge node [above] {\(\mf_1^2\)} (b2);
\draw[thick]  (b2) edge node [above] {\(\mf_2^2\)} (c2);
\draw[thick]  (c2) edge node [above] {\(\mf_3^2\)} (d2);
\draw[thick]  (d2) edge node [above] {\(\mf_4^2\)} (e2);
\draw[thick]  (a1) edge node [left] {\(\me_1\)} (a2);
\draw[thick]  (b1) edge node [left] {\(\me_2\)} (b2);
\draw[thick]  (c1) edge node [left] {\(\me_3\)} (c2);
\draw[thick]  (d1) edge node [left] {\(\me_4\)} (d2);
\draw[thick]  (e1) edge node [left] {\(\me_5\)} (e2);
\draw[thick,dashed]  (e1) edge (f1);
\draw[thick,dashed]  (e2) edge (f2);
\draw[fill] (a1) circle (1.75pt);
\draw[fill] (b1) circle (1.75pt);
\draw[fill] (c1) circle (1.75pt);
\draw[fill] (d1) circle (1.75pt);
\draw[fill] (e1) circle (1.75pt);
\draw[fill] (a2) circle (1.75pt);
\draw[fill] (b2) circle (1.75pt);
\draw[fill] (c2) circle (1.75pt);
\draw[fill] (d2) circle (1.75pt);
\draw[fill] (e2) circle (1.75pt);
\end{tikzpicture}
\caption{The infinite ladder graph.}\label{fig:ladder-graph}
\end{figure}
Suppose that the edge lengths of \(\Graph\) are given by \(\ell_{\mf_n^1}=\ell_{\mf_n^2}=\frac{1}{2^n}\) and \(\ell_{\me_n}=1\) for all \(n\in\mathbb N\). Then, the diameter of \(\Graph\) is given by
	\[
		\diam(\Graph) = \ell_{\me_1}+\sum_{n=1}^\infty \ell_{\mf_n^1} = \sum_{n=0}^\infty \frac{1}{2^n} = 2.
	\]
Now, let \((\Graph_n)_{n\in\mathbb N}\) be the compact exhaustion of \(\Graph\) given by
	\[
		\Graph_n = \bigcup_{k=1}^{n+1}\me_k \cup \bigcup_{k=1}^{n}\mf_k^1 \cup \bigcup_{k=1}^{n-1}\mf_k^2
	\]
	(see \autoref{fig:ladder-graph-exhaustion}).
\begin{figure}[ht]
\begin{minipage}[t]{0.2\textwidth}
\centering
\begin{tikzpicture}
\coordinate (a1) at (0,0);
\coordinate (b1) at (2,0);
\coordinate (a2) at (0,2);
\coordinate (b2) at (2,2);
\draw[thick]  (a1) edge node [below] {\(\mf_1^1\)} (b1);
\draw[thick]  (a1) edge node [left] {\(\me_1\)} (a2);
\draw[thick]  (b1) edge node [left] {\(\me_2\)} (b2);
\draw[fill] (a1) circle (1.75pt);
\draw[fill] (b1) circle (1.75pt);
\draw[fill] (a2) circle (1.75pt);
\draw[fill] (b2) circle (1.75pt);
\end{tikzpicture}
\end{minipage}
\begin{minipage}[t]{0.3\textwidth}
\centering
\begin{tikzpicture}
\coordinate (a1) at (0,0);
\coordinate (b1) at (2,0);
\coordinate (c1) at (4,0);
\coordinate (a2) at (0,2);
\coordinate (b2) at (2,2);
\coordinate (c2) at (4,2);
\draw[thick]  (a1) edge node [below] {\(\mf_1^1\)} (b1);
\draw[thick]  (b1) edge node [below] {\(\mf_2^1\)} (c1);
\draw[thick]  (a2) edge node [above] {\(\mf_1^2\)} (b2);
\draw[thick]  (a1) edge node [left] {\(\me_1\)} (a2);
\draw[thick]  (b1) edge node [left] {\(\me_2\)} (b2);
\draw[thick]  (c1) edge node [left] {\(\me_3\)} (c2);
\draw[fill] (a1) circle (1.75pt);
\draw[fill] (b1) circle (1.75pt);
\draw[fill] (c1) circle (1.75pt);
\draw[fill] (d1) circle (1.75pt);
\draw[fill] (a2) circle (1.75pt);
\draw[fill] (b2) circle (1.75pt);
\draw[fill] (c2) circle (1.75pt);
\end{tikzpicture}
\end{minipage}
\begin{minipage}[t]{0.4\textwidth}
\centering
\begin{tikzpicture}
\coordinate (a1) at (0,0);
\coordinate (b1) at (2,0);
\coordinate (c1) at (4,0);
\coordinate (d1) at (6,0);
\coordinate (a2) at (0,2);
\coordinate (b2) at (2,2);
\coordinate (c2) at (4,2);
\coordinate (d2) at (6,2);
\draw[thick]  (a1) edge node [below] {\(\mf_1^1\)} (b1);
\draw[thick]  (b1) edge node [below] {\(\mf_2^1\)} (c1);
\draw[thick]  (c1) edge node [below] {\(\mf_3^1\)} (d1);
\draw[thick]  (a2) edge node [above] {\(\mf_1^2\)} (b2);
\draw[thick]  (b2) edge node [above] {\(\mf_2^2\)} (c2);
\draw[thick]  (a1) edge node [left] {\(\me_1\)} (a2);
\draw[thick]  (b1) edge node [left] {\(\me_2\)} (b2);
\draw[thick]  (c1) edge node [left] {\(\me_3\)} (c2);
\draw[thick]  (d1) edge node [left] {\(\me_4\)} (d2);
\draw[fill] (a1) circle (1.75pt);
\draw[fill] (b1) circle (1.75pt);
\draw[fill] (c1) circle (1.75pt);
\draw[fill] (d1) circle (1.75pt);
\draw[fill] (a2) circle (1.75pt);
\draw[fill] (b2) circle (1.75pt);
\draw[fill] (c2) circle (1.75pt);
\draw[fill] (d2) circle (1.75pt);
\end{tikzpicture}
\end{minipage}

\caption{From left to right, the graphs \(\Graph_1\), \(\Graph_2\) and \(\Graph_3\).}\label{fig:ladder-graph-exhaustion}
\end{figure}
Then, the diameter of graph \(\Graph_n\) is given by
\[
	\diam(\Graph_n) = \ell_{\me_1}+\sum_{k=1}^n \ell_{\mf_k^1} + \ell_{\me_{n+1}} =2 + \sum_{k=1}^n \frac{1}{2^k}.
\]
Consequently, we have
\[
	\diam(\Graph)=2 < 3 =\lim_{n\rightarrow \infty}\diam(\Graph_n).
\]
In particular, the graph \(\Graph\) and the compact exhaustion \((\Graph_n)_{n\in\mathbb N}\) indeed provide an example for strict inequality in \eqref{eq:InequalityInfiniteDiameter}. We leave it as an open question if metric graphs exist such that the inequality in~\eqref{eq:InequalityInfiniteDiameter} is strict for \textit{any} compact exhaustion.
\end{exa}

\begin{prop}\label{PropDiameterConvergence}
Let $\mGraph$ be a locally finite, connected metric graph. If $\diam(\mGraph) < \infty$ assume in addition that the Betti number $\beta$ is finite. Then there exists a compact exhaustion $(\mGraph_n)_{n \in \N}$ such that $\diam(\mGraph_n) \to \diam(\mGraph)$.
\end{prop}

We again recall that, by Remark~\ref{rem:redu}, finiteness of the Betti number is equivalent to the condition that there is a compact subgraph of $\Graph$ containing all cycles.
In the case that $\diam(\mGraph) < \infty$, our proof will show that $\diam(\mGraph_n) \leq \diam(\mGraph)$ for sufficiently large $n$.

\begin{proof}
It suffices to consider the case that $\mGraph$ is infinite and, due to \eqref{EquationInfiniteDiameter}, that $\diam(\mGraph) < \infty$. 
We fix $\mv\in\mV$, and consider the compact exhaustion $\mGraph_{\mv,n}$ of Example~\ref{ExampleApprox}. 
For $n \in \N$ large enough that all cycles of $\mGraph$ are contained in $\mGraph_{\mv,n}$, as noted in Remark~\ref{RemarkAlmostTree}, $\mGraph \setminus \mGraph_{\mv,n}$ is a disjoint union of trees, each attached to $\mGraph_{\mv,n}$ at a single vertex.

We claim that, for such $n$, 
	\begin{equation}\label{EquationDiameterExhaustion}
		\dista_{\mGraph_{\mv,n}} (x,y) = \dista_{\mGraph} (x,y) \qquad \text{for all } x,y \in \mGraph_{\mv,n}.
	\end{equation}
Indeed, let $P$ be a walk in $\mGraph$ from $x$ to $y$, see Definition~\ref{DefinitionBetti}. 
It suffices to find a walk $P_n$ in $\mGraph_{\mv,n}$ connecting $x$ and $y$ which is no longer than $P$, since \eqref{EquationDiameterExhaustion} will then follow immediately. 
In fact, if the walk $P$ is \emph{not} contained in $\mGraph_{\mv,n}$, then part of it must lie in one or more of the trees of which $\mGraph\setminus \mGraph_{\mv,n}$ consists. Suppose this tree is attached to $\mGraph_{\mv,n}$ at a single vertex $\mw$; then $P$ must pass through $\mw$ twice. Cut out the part of $P$ beyond $\mw$, glue the two remaining parts of the walk together, and repeat for every connected component of $\mGraph \setminus \mGraph_{\mv,n}$ through which $P$ passes. The new walk $P_n$ lies in $\mGraph_{\mv,n}$, still connects $x$ and $y$, and has shorter length. This proves the claim.

It follows from \eqref{EquationDiameterExhaustion} and the fact that $\mGraph_{\mv,n} \subset \mGraph$ that, for all $n$ for which \eqref{EquationDiameterExhaustion} holds,
\begin{displaymath}
	\diam(\mGraph_{\mv,n}) = \sup_{x,y \in \mGraph_{\mv,n}} \dista_{\mGraph_{\mv,n}} (x,y) = \sup_{x,y \in \mGraph_{\mv,n}} \dista_{\mGraph} (x,y)
	\leq \sup_{x,y \in \mGraph} \dista_{\mGraph} (x,y) = \diam(\mGraph).
\end{displaymath}
Combining this with the result of Proposition~\ref{PropDiameterBelow} yields the conclusion.
\end{proof}

\subsection{Lebesgue and Sobolev spaces}
	\label{subsec:Sobolev_spaces}

	We next introduce function spaces on a metric graph $\Graph$,  for which embedding theorems will be discussed in Section~\ref{sec:embedding}.
	
	Any metric graph $\Graph$ has both a topological and a measure theoretical structure.
	This immediately defines the space of continuous functions and Lebesgue spaces; we refer to \cite{Mug19} and \cite[Sections~2.3 and 3.1]{KosMugNic22} for the details of these function spaces.

\begin{defi}\label{DefinitionL2}
Define
\begin{displaymath}
	L^2 (\mGraph) := \left\{ f \in \bigoplus_{\me \in \mE} L^2(\me) : \sum_{\me \in \mE} \|f_\me\|_{L^2(\me)}^2 < \infty \right\},
\end{displaymath}
equipped with the canonical inner product and norm, and define
\begin{displaymath}
	L_c^2(\mGraph) := \left\{ f \in L^2(\mGraph) : f(x) = 0 \text{ almost everywhere outside a compact subset of } \mGraph \right\}.
\end{displaymath}
\end{defi}

The other $L^p$-spaces, $p \in [1,\infty]$, are defined analogously.
In view of Remark~\ref{rem:redu}, any function in $L_c^2(\mGraph)$ is identically zero outside of a finite set of edges of $\mGraph$.

\begin{defi}\label{DefinitionC}
Denote by $C(\mGraph)$ the set of all functions $f: \mGraph \to \R$ which are continuous with respect to the canonical metric on $\mGraph$, and by
\begin{displaymath}
	C_c(\mGraph) := \left\{ f \in C(\mGraph): f(x) = 0 \text{ outside a compact subset of } \mGraph \right\}
\end{displaymath}
the set of continuous functions of compact support.
\end{defi}

Again, $C_c(\Graph)$-functions are necessarily supported on a finite set of edges, and, if $\mGraph$ has any ends, then $C_c(\Graph)$ is not a closed space.
Denoting by $\overline{\mGraph}$ the metric space which is the union of $\mGraph$ with its ends as in Section~\ref{subsec:notation}, we also write $C(\overline{\mGraph})$ for the complete space of all continuous functions on $\overline{\mGraph}$.
This is canonically identified with the space of continuous functions on $\mGraph$ that can be continuously extended to the ends of the graph.

\begin{defi}\label{H1def}
Denote by $H^1(\mGraph)$ the Sobolev space
\begin{displaymath}
	H^1(\mGraph) := \left\{ f \in C(\mGraph): f_\me \in H^1(\me) \text{ for all $\me \in \mE$, and } \sum_{\me \in \mE} \|f_\me\|_{H^1(\me)}^2 < \infty \right\},
\end{displaymath}
equipped with the canonical inner product and norm. Here, the space $H^1(\me)$ consists exactly of the absolutely continuous functions on $\me$ whose distributional derivative is in $L^2(\me)$.

Replacing $L^2$-spaces by $L^p$-spaces in the definition, we also obtain the Sobolev spaces $W^{1,p} (\mGraph)$ for $p\geq 1$; 
in particular, $W^{1,2} (\Graph) = H^1 (\Graph)$.
It is not hard to show that, for any $f \in H^1(\mGraph)$, there is a canonical extension of $f$ to the ends of $\mGraph$ (see \cite[Definition~3.4]{KosMugNic22}). 
In particular, if $\overline{\Graph}$ is compact (see Corollary~\ref{cor:all_graphs_end_finite_volume} for a sufficient condition), then by~\cite[Lemma~3.2]{KosMugNic22}
 $f \in H^1(\mGraph)$ can be identified with a function on $\overline{\mGraph}$, and up to a canonical identification, we have $H^1 (\mGraph) \subset C(\overline{\mGraph})$.

We also define the (not necessarily closed) subspace of $H^1(\Graph)$ of $H^1$-functions of compact support,
\begin{equation}\label{H1c}
	H^1_c(\mGraph) := \left\{ f \in H^1(\mGraph) : f(x) = 0 \text{ outside a compact subset of } \mGraph \right\}
\end{equation}
and its completion in $H^1$,
\begin{equation}\label{H10-original}
	H^1_0 (\mGraph) := \overline{H^1_c (\mGraph)}^{\|\cdot\|_{H^1(\mGraph)}}.
\end{equation}
\end{defi}
The issue whether $H^1_0(\Graph)=H^1(\Graph)$ is intimately related to the essential self-adjointness of the Laplacian $\Delta_{{\Graph}|_{L^2_c}}$.
This will be briefly discussed in Section~\ref{subsec:laplacian-inf}, and is addressed in more detail in~\cite{KosMugNic22}.

\begin{rem}\label{Dirichletknoten}
\begin{enumerate}
	\item
	It is easy to see (and well known in the case of finite metric graphs) that, up to isometric isomorphism, the spaces $H^1(\Graph)$ and $L^2(\Graph)$ are not modified upon inserting or deleting dummy vertices.
	\item
	Similarly to (1), we observe that formally modifying $\Graph$ by identifying two or more Dirichlet vertices does not affect the spaces $H^1_0(\Graph)$ or $L^2(\Graph)$.
	\item
	One has
	\begin{displaymath}
		H^1_0(\mGraph) = \{ f \in H^1(\mGraph) : f(\gamma) =0 \text{ for all } \gamma \in \mathfrak{C}(\mGraph) \},
	\end{displaymath}
	that is, $H^1_0 (\mGraph)$-functions satisfy a Dirichlet (zero) condition at every end, see~\cite[Theorem 3.12]{KosMugNic22}. 
	\item
	Correspondingly, it will also be of interest to consider Sobolev spaces associated with $\mGraph$ where we impose Dirichlet conditions on a subset of $\mathfrak{C}(\mGraph)\cup\mV$; for some subset $\Dir \subset\mathfrak{C}(\mGraph)\cup\mV$ we define the space
		\[
			H^1_0(\mGraph;\Dir) :=\{f \in H^1(\mGraph) : f(\mv) =0 \text{ for all } \mv \in \Dir \}.
		\]
	We regard this as imposing a Dirichlet condition at each vertex and each end belonging to $\Dir$.
	\item
	A unified way to prescribe \emph{Dirichlet conditions at vertices}, is as follows: remove $\mv$ and appropriately add infinitely many dummy vertices on all edges incident to it.
	This way, $ \mv $ is replaced by $\deg(\mv)$ many ends $ \gamma $, and every $f \in H^1_0(\mGraph)$ must vanish at $\mv$. 
\end{enumerate}
\end{rem}

\subsection{Laplace-type operators on infinite graphs}\label{subsec:laplacian-inf}

Following~\cite{KosMugNic22}, we recall two natural realisations  $\mathcal{H}_{\mathrm{N}}$ and  $\mathcal{H}_{\mathrm{F}}$ of the Laplacian on a locally finite graph $\mGraph$.
 We start by defining the maximal Laplacian, which does not see the topological structure of the graph.

\begin{defi}\label{DefinitionLaplace}

The maximal Laplacian is defined by

\[ \mathcal{H}_{\max} = \bigoplus_{\me \in \mE} 
- \frac{\ud^2}{\ud x_\me^2}
, \qquad \dom(\mathcal{H}_{\max}) = \bigoplus_{\me \in \mE} H^2(\me). \]
\end{defi}

We will always restrict this operator to spaces satisfying standard vertex conditions: firstly, any function $f$ in the operator domain should be in $C(\mGraph)$; secondly, it should satisfy the Kirchhoff condition, where the sum of the normal derivatives of $f$ at every vertex should be zero:
 \[
\sum_{\me \in \mE_{\mv}} \frac{\partial}{\partial \nu_\me} f(\mv) = 0 \qquad \text{for all }\mv \in \mV. \]
Imposing only these two conditions leads to the \emph{maximal Kirchhoff Laplacian} on $ \mGraph $: 
this is the restriction of $\mathcal{H}_{\max}$ to the domain
\[\dom(\mathcal{H}) = \left\{ f \in \dom(\mathcal{H}_{\max}) \cap L^2(\mGraph): f \in C(\mGraph) \text{ and }\sum_{\me \in \mE_\mv} \frac{\partial}{\partial \nu_\me} f(\mv) = 0
\text{ for all } \mv \in \mV \right\}.\]

In order to construct self-adjoint restrictions of the maximal Kirchhoff Laplacian, it is natural to resort to quadratic forms:
\begin{align}
\label{NeumannForm} \form_{\mathrm{N}}[f] &:= \int_{\mGraph} | f'(x) |^2 \ud x,  &\dom[t_{\mathrm{N}}] := H^1(\mGraph), \\
\label{FriedrichsForm} \form_{\mathrm{F}}[f] &:= \int_{\mGraph} | f'(x) |^2 \ud x,  &\dom[t_{\mathrm{F}}] := H^1_0(\mGraph),
\end{align}
from which the Kirchhoff-type conditions at the vertices arise naturally.

\begin{defi}\label{def:neumann-friedrichs}
	The \emph{Neumann realisation} $\mathcal{H}_{\mathrm{N}}$ is the self-adjoint operator on $L^2(\mGraph)$ associated with $\form_{\mathrm{N}}$, 
	and the \emph{Friedrichs realisation} $\mathcal{H}_{\mathrm{F}}$ is the self-adjoint operator on $L^2(\mGraph)$ associated with $\form_{\mathrm{F}}$.
\end{defi}

It can be shown, see~\cite[Corollary~6.7]{KosMugNic22} and also Remark~\ref{Dirichletknoten}.(3), that on metric graphs with a finite number of ends, each of which of finite volume, $\mathcal{H}_{\mathrm{N}}$ corresponds to imposing Neumann boundary conditions at all ends while the Friedrichs realisation $\mathcal{H}_{\mathrm{F}}$ imposes Dirichlet boundary conditions at all ends.

\begin{rem}\label{rem:mixed-Neumann-Friedrich}
We can define
mixed versions of the Neumann and Friedrichs realisations: given $\Dir \subset \mV \cup \mathfrak C(\Graph)$, take the quadratic form given by
\begin{align*}
\form_\Dir[f] &:= \int_{\mGraph} | f'(x) |^2 \ud x, & \dom[\form_\Dir] : =H^1_0(\mGraph;\Dir).
\end{align*}
We introduced $H^1_0(\mGraph;\Dir)$ in Remark~\ref{Dirichletknoten}.(4): since it
 is a closed subspace of $H^1(\mGraph)$ that is dense in $L^2(\mGraph)$, the form $\form_\Dir$ induces a self-adjoint operator $\mathcal H_\Dir$ in $L^2(\mGraph)$. Roughly speaking, for all functions in the domain of  $\mathcal H_\Dir$ we are imposing Dirichlet-type conditions at \textit{some} ends and/or at \textit{some} vertices, along with standard conditions at all other vertices and Neumann conditions at all other ends. We will sometimes sloppily refer to $\mathcal H_\Dir$ as the \emph{Laplacian with mixed Dirichlet/Neumann conditions}.
\end{rem}

\section{Embedding theorems and discreteness of the spectrum}
\label{sec:embedding}

\subsection{Embedding theorems for Sobolev spaces}
	\label{subsec:Sobolev_embedding}

This section is about criteria for discreteness of the spectrum of the operator $\mathcal{H}_{\mathrm{N}}$ and all self-adjoint Laplacian realisations dominating $\mathcal{H}_{\mathrm{N}}$ in the sense of forms, including $\mathcal{H}_{\mathrm{F}}$. 
Since discreteness of the spectrum of self-adjoint realisations is equivalent to compactness of the embedding of their form domains in $L^2(\mGraph)$ this subsection can also be understood as discussing Sobolev embedding theorems.
For (locally finite) infinite metric graphs with compact metric completion it 
has been known since \cite[Theorem~2.2]{Car00} that
$H^1(\mGraph)$ embeds compactly into $C_b(\overline{\mGraph})$, the space of bounded continuous functions over the metric completion of $\mGraph$. We stress that, however, this does not necessarily imply a compact embedding into $L^2(\Graph)$, unless $\overline{\mGraph}$ has finite volume (for a counterexample, see the diagonal comb graph in \ref{thm:comb}.(2)).

In particular, the resolvent of $\mathcal H_{\mathrm{N}}$ and hence -- by pointwise domination of the integral kernels -- also of  $\mathcal H_{\mathrm{F}}$ or any Markovian realisation -- will be trace class operators, see~\cite[Theorem~5.1]{KosMugNic22}. However, this finite  volume criterion is insufficient to \emph{characterise} discreteness of the spectrum: there are metric graphs of infinite  volume on which $\mathcal{H}_{\mathrm{F}}$ and even $\mathcal{H}_{\mathrm{N}}$ have purely discrete spectrum: Non-trivial examples are given in~\cite[Theorem~1.3.5]{AliNic93} and \cite[Theorem~4.1]{Sol04}.

This is, more generally, remedied by the next criterion which also constitutes our first main result.
It is inspired by~\cite{HanHol10}, and can be understood as a Kolmogorov--Riesz-type result characterising form domains which are compactly embedded in $ L^2(\Graph) $.

\begin{theo}\label{theo:duefel}
Let $\Graph $ be a locally finite, connected metric graph, and let $K$ be a closed subspace of $ H^1(\Graph)$.
Then the embedding of $K $ into $L^2(\Graph)$ is compact if and only if
for all $\varepsilon>0$ there is a finite subgraph $\Graph_c$ of $\Graph$ such that
\begin{equation}
	\label{eq:criterium_compactness}
	\|f\|_{L^2\left((\Graph_c)^\complement \right)}
	\le \varepsilon 
	\quad 
	\text{for all $ f \in K$ with} 
	\quad 
	\lVert f \rVert_{H^1(\Graph)} \leq 1,
\end{equation}
where $(\Graph_c)^\complement$ denotes the complement of $\Graph_c$ within $\Graph$.
\end{theo}

Clearly, Theorem~\ref{theo:duefel} contains the existing criterion of finite  volume, since in this case, $H^1(\Graph) \hookrightarrow L^\infty(\Graph)$, and thus $\|f\|_{L^2((\Graph_c)^\complement)} \leq |(\Graph_c)^\complement|^{1/2} \|f\|_{L^\infty ((\Graph_c)^\complement )}$ can be made arbitrarily small independently of $f$ by choosing $|(\Graph_c)^\complement|$ is small enough.
The usefulness of Theorem~\ref{theo:duefel} will be demonstrated in Theorem~\ref{thm:comb}.

\begin{proof}
$(\Longleftarrow)$
Let $\varepsilon>0$, take a finite subgraph $\Graph_c \subset \Graph$ as in the theorem, and denote by $ B $ the unit ball within $K$.
Construct a new graph $\Graph^*$ from $\Graph$ by doubling all edges in $\Graph_c$, that is, replacing each $\me$ in $\Graph_c$ with a pair $(\me', \me'')$ of identical edges, and keeping edges in $(\Graph_c)^\complement$.
Let $ B^* $ be the set of functions $ f^* \in H^1(\Graph^*) $ such that there exists $f \in B$ with 
	\[
	\label{eq:criterion_compactness}
	f^*_{\vert \me'} = f^*_{\vert \me''} = f_{\vert \me} 
	\quad
	\text{for all $\me \subset \Graph_c$, and}
	\quad
	f^*_{(\Graph_c)^\complement}
	=
	f_{(\Graph_c)^\complement}.
	\]
The subgraph $ \Graph_c^* :=\mGraph^*\setminus (\Graph_c)^\complement$ can be identified with an Eulerian walk
of finite length $ p $ which allows to identify $ f^*_{\vert \Graph_c^* } \in B^* $ with a function $ \hat{f} \in H^1 ([0,p]) $. 
On $\Graph_c^* $ we can then define 
	\[  
	f^*(x+h) =
	\begin{cases} 
		\hat{f}(x+h) & 
		\quad 
		\text{if}\ x \in \Graph_c^*,\ \text{and}\ x + h \in [0,p], \\
		0 & 
		\quad \text{if}\ x \in \Graph_c^*,\ \text{and}\ x + h \notin [0,p]. 
	\end{cases} 
	\]
By uniform continuity of $ \hat{f} $ on $ [0,p] $ there is $ \delta_0 >0 $ so that for each $\delta \leq \delta_0$ and $ h < \delta$ we have 
	\begin{equation}
	\label{eq:consequence_of_uniform_continuity}
		\int_{\Graph_c^*} \lvert f^*(x+h) - f^*(x) \rvert^2 \mathrm{d} x < \varepsilon^2. 
	\end{equation}
We cover $ [0, p ] $ with $N$ non-overlapping intervals of length $\delta$,
	\[
	[0, p]
	=
	\bigcup_{j = 1}^N \overline{I_j},
	\quad
	\text{where}
	\quad
	I_j \cap I_k = \emptyset,
	\quad
	\text{for $k \neq j$}, 
	\quad
	\text{and}
	\quad
	\lvert I_j \rvert = \delta 
	\quad
	\text{for all $j \in \{1, \dots, N \}$}.
	\]
This allows us to define an orthogonal projection $P \colon K \to L^2(\Graph^*)$
	\[
	P(f^*)(x) 
	:= 
	\begin{cases}
		\frac{1}{\delta} 
		\int_{I_j} f^*(x) \mathrm{d} x 
		\quad
		&\text{if $x \in \Graph_c^*$, and $x \in I_j$ on the Eulerian walk},\\
		0
		\quad
		&
		\text{if $x \in (\Graph_c)^\complement$}.
	\end{cases}
	\]

We have
\begin{align}
	\label{eq:split_subgraph_and_rest}
	\lVert f^* - Pf^* \rVert_{L^2({\Graph^*})}^2 
	&= 	
	\int_{\Graph_c^\complement} 
	\lvert f^* - Pf^* \rvert^2 
	+ 	
	\int_{\Graph^*_c} \lvert f^* - Pf^* \rvert^2 ,
\end{align}
and note that the first summand in~\eqref{eq:split_subgraph_and_rest} is at most $ \varepsilon^2 $ by assumption. 
The second term in~\eqref{eq:split_subgraph_and_rest}  can be considered as an integral along the Eulerian walk and estimated using Jensen's inequality by
	\begin{align*}
	\sum_{j=1}^N \int_{I_j} 
	\left \lvert 
		\frac{1}{\delta} 
		\int_{I_j} \left(f^*(x) - f^*(y)\right) \mathrm{d}y 
	\right \rvert^2 \mathrm{d}x
	\leq
	\sum_{j=1}^N 
	\int_{I_j}
	\frac{1}{\delta} 
	\int_{I_j} 
	\left \lvert 
		f^*(y) - f^*(x) 
	\right \rvert^2 
	\mathrm{d}y\
	\mathrm{d}x .
	\end{align*}
We substitute $h = y-x $ and emphasise that $ h \in (-\delta, \delta)$. 
Thus,~\eqref{eq:consequence_of_uniform_continuity} implies
\begin{align*}
	\lVert f^* - Pf^* \rVert_{L^2(\mathcal{G^*})}^2 
	&\leq 
	\varepsilon^2  
	+ 
	\sum_{j=1}^N 
	\int_{I_j}
	\frac{1}{\delta} 
	\int_{-\delta}^\delta 
	\left \lvert 
		f^*(x+h) - f^*(x) 
	\right \rvert^2 
	\mathrm{d}h\
	\mathrm{d}x 
	\\
	&=
	\varepsilon^2 
	+ 
	\frac{1}{\delta} 
	\int_{-\delta}^\delta
	\int_{\Graph_c^*}
	\lvert 
		f^*(x+h) - f^*(x) 
	\rvert^2 
	\mathrm{d}h\
	\mathrm{d}x  
	\leq 3 \varepsilon^2.
\end{align*}
Using the triangle inequality,
	\[
	\lVert 
	f^* 
	\rVert_{L^2(\mathcal{G^*})} 
	\leq 
	\sqrt{3}\varepsilon + \lVert Pf^* \rVert_{L^2(\mathcal{G^*})}. 
	\]
Now, if $ f^*,g^* \in B^* $ with $ \lVert Pf^* - Pg^* \rVert_{L^2(\mathcal{G^*})} < \varepsilon $, then we have 
	\[
	\lVert 
		f^* - g^* 
	\rVert_{L^2(\mathcal{G^*})} 
	< 
	(\sqrt{3}+1)\varepsilon. 
	\]
Furthermore, $ P $ is a bounded operator, and $ B^* $ is bounded, thus $ P(B^*) $ itself is bounded. Finally, since $P$ is of finite rank, it is totally bounded. 
By~\cite[Lemma 1]{HanHol10}, we find that $ B^* $ is totally bounded in $ L^2(\Graph^*) $, and thus the same is true of $ B $ in $ L^2(\Graph) $.
\\
($\Longrightarrow$) Conversely, assume that $K$ is compactly embedded in $ L^2(\Graph) $. 
Then the unit ball $ B \subset K $ is totally bounded in $ L^2(\Graph) $, and for every $ \varepsilon > 0 $ there is a finite $ \varepsilon$-cover $ \{ U_1, \ldots U_n \} $ with central points $ g_1, \dots, g_n \in B $. 
For each $ g_i $ there is a finite subgraph $ \Graph_i $ such that $\lVert g_i \rVert_{L^2((\Graph_i)^\complement)} < \varepsilon$.
	Thus, there is a finite and connected subgraph $ \Graph_c $ that fulfils the conditions for all $ g_i $ with $ 1 \leq i \leq n $. 
	Now let $ f \in B $. 
	There is $ g_i $ with $ \lVert f - g_i \rVert_{L^2(\Graph) } < \varepsilon $, whence by the triangle inequality
	\begin{align*}
		\lVert f \rVert_{L^2((\Graph_c)^\complement)} 
		&\leq 
		\lVert 
		f - g_i 
		\rVert_{L^2((\Graph_c)^\complement)}
		+
		\lVert 
		g_i 
		\rVert_{L^2((\Graph_c)^\complement)}
		<
		2 \varepsilon
		.
		\qedhere
	\end{align*}

\end{proof}

Next, we obtain a criterion on compactness of embeddings of Sobolev spaces on \emph{trees}.
It is a generalisation of~\cite[Lemma~8.1]{KosNic19} where infinite trees \textit{without degree one vertices} are considered. 
We also mention that the metric condition in Proposition~\ref{prop:length_epsilon_compact} is already known to play a role for spectral properties of infinite quantum graphs with $\delta$-couplings, see~\cite[Theorem 3.5]{ExnKosMal18}.

\begin{prop}
	\label{prop:length_epsilon_compact}
	Let $\Graph$ be a locally finite metric tree and 
	\[
	\Dir:=\mathfrak{C}(\Graph)\cup\{\mv\in\mV : \deg(\mv)=1\}
	.
	\]
	If for all $\varepsilon > 0$ there are only finitely many edges of length larger than $\varepsilon$, then the embedding of $H_0^1(\Graph, \Dir)$ into $L^2(\Graph)$ is compact and $\mathcal{H}_\Dir$ has purely discrete spectrum.
\end{prop}

We stress that $\Graph$ has not been assumed to have finite diameter.

\begin{proof}
On infinite trees without degree one vertices (``leaves''), the statement follows from 
\cite[Lemma~8.1]{KosNic19}

It remains to prove the statement in the presence of leaves.
	
	Compactness of the embedding is equivalent to 
	\[
	\lim_{k \to \infty} \lambda_k(\Graph, \Dir) = \infty,
	\quad
	\text{where}
	\quad
	\lambda_k(\Graph, \Dir) 
	=
	\inf_{\substack{X \subset H_0^1(\Graph, \Dir) \\ \operatorname{dim} X = k}}
	\sup_{\substack{\phi \in X \\ \lVert \phi \rVert_{L^2(\Graph)} = 1}}
	\lVert \phi' \rVert_{L^2(\Graph)}.
	\]
	Enumerate the at most countably many leaves by $\mv_j$ and turn the tree $\Graph$ into a leafless tree $\Graph^+$ by taking copies of a leafless tree of finite diameter, scaling it by $j^{-1}$, and attaching it to $\mv_j$.
	Now, $\Graph^+$ is a leafless tree satisfying the conditions of the proposition. 
	Furthermore, since $\Graph \subset \Graph^+$, we have $H_0^1(\Graph, \Dir) \subset H^1_0(\Graph^+)$, which implies 
	\[
	\infty = \lim_{k \to \infty} \lambda_k(\Graph^+) \leq \lim_{k \to \infty} \lambda_k(\Graph, \Dir).
	\qedhere
	\]
\end{proof}

\subsection{Spectral phase transition on graphs of infinite volume}\label{H10yesH1not}

This subsection discusses \textit{diagonal combs}, a class of examples which serve two purposes:
On the one hand, they illustrate that on infinite metric graphs even the type of spectrum can depend in a subtle way on the boundary conditions. 
On the other hand, the diagonal combs underline the usefulness of Theorem~\ref{theo:duefel} by providing a parameter-depending family of graphs $(\mGraph_\alpha)_{\alpha > 0}$ on which the spectrum of $\mathcal{H}_{\mathrm{N}}$ exhibits a phase transition from purely discrete to non-empty essential spectrum, which surprisingly occurs among graphs of infinite volume.
We comment on this phenomenon and compare it to the Laplacian on so-called horn-shaped domains in Remark~\ref{rem:phase_transition} below.

Furthermore, we identify the driving mechanism behind this phase transition (the growth rate of volumes of annuli around a graph end) and formulate a corresponding general statement in Theorem~\ref{theo:discrete_spectrum_annuli}.

\begin{defi}\label{def:comb}
For $\alpha > 0$, we define the ``diagonal comb'' metric graph $\Graph_\alpha$ by taking the interval $(0,1]$ (``horizontal shaft''), putting a vertex on every point $\frac{1}{n^\alpha}$, $n \in \mathbb{N}$,  and attaching to it an edge (``tooth'') of length $\frac{1}{n^\alpha}$.
\end{defi}

\begin{figure}
\begin{center}
	\begin{tikzpicture}[scale = 4]
		\draw[thick] (0,0) -- (1,0);
		
		\foreach \x in {1,...,50}
		{
		\draw[thick] ({1/sqrt(\x)},0) -- ({1/sqrt(\x)},{1/sqrt(\x)});
		}
		\draw[fill] (0,0) -- (.15,0) -- (.15,.15) -- (0,0);
		
		\draw(0,-.1) node {$0$};
		\draw(1,-.1) node {$1$};
		
		\begin{scope}[xshift = 1cm]
		  \draw[very thick] (-.03,-.03) -- (.03,.03);
		  \draw[very thick] (-.03,.03) -- (.03,-.03);
		\end{scope}
		\begin{scope}[xshift = 1cm, yshift = 1cm]
		  \draw[very thick] (-.03,-.03) -- (.03,.03);
		  \draw[very thick] (-.03,.03) -- (.03,-.03);
		\end{scope}
		
	\end{tikzpicture}
\caption{The diagonal comb graph $\mGraph_{\frac{1}{2}}$ and its two centre vertices, see Definition \ref{def:centre_vertex} and Remark~\ref{rem:two:centres}.}
\label{fig:no_compact_embedding}
\end{center}
\end{figure}

Figure~\ref{fig:no_compact_embedding} contains an illustration in the case $\alpha=\frac{1}{2}$. 
The diagonal comb $\mGraph_\alpha$ has, for any $\alpha \in (0,1]$, one topological end, diameter equal to 2 and volume $1 + \sum\limits_{n=1}^{\infty} \frac{1}{n^{\alpha}}$.
Note that the larger $\alpha$, the sparser the teeth become.
In particular, the graph has finite volume if (and only if!) $\alpha > 1$ in which case $\mathcal H_{\mathrm{N}}$ (or any Markovian realisation of $\mathcal H$) will have purely discrete spectrum by~\cite[Theorem~2.2]{Car00}.

\begin{theo}
\label{thm:comb}
For the diagonal comb metric graph $\Graph_\alpha$ we have:
\begin{enumerate}
	\item For all $\alpha > \frac{1}{2}$, $H^1 (\Graph_\alpha)$ is compactly embedded in $L^2(\Graph_\alpha)$. In particular, $\mathcal H_{\mathrm{N}}$ has compact resolvent and purely discrete spectrum.
	\item If $\alpha \in (0,\frac{1}{2}]$, then there is an $L^2(\Graph_\alpha)$-orthonormal sequence of $H^1(\Graph_\alpha)$-functions with uniformly bounded $H^1(\Graph_\alpha)$ norm. In particular, $\mathcal H_{\mathrm{N}}$ has non-empty essential spectrum.
	\item For all $\alpha > 0$, if we set $\Dir$ to be the union of $\{0\}$ with the set of all tips of the teeth, then $H_0^1(\Graph_\alpha, \Dir)$ is compactly embedded in $L^2(\Graph_\alpha)$. 
	Consequently, $\mathcal H_\Dir$ has compact resolvent and purely discrete spectrum.
\end{enumerate}
\end{theo}

Before giving the proof of this result, a few observation are due.
Note that the phase transition happens at $\alpha = 1/2$, that is, among graphs of infinite volume (the transition from infinite to finite volume being at $\alpha = 1$); see also Remark~\ref{rem:critical-spectrum}.
Also, all $\mGraph_\alpha$ are trees with only a finite number of edges of length larger than $\varepsilon$ for any $\varepsilon > 0$ -- the criterion for discreteness of the spectrum of the Friedrichs realisation in Proposition~\ref{prop:length_epsilon_compact}.
Thus, Theorem~\ref{thm:comb} also demonstrates that this criterion for discreteness of the Friedrichs realisation on trees does not extend to the Kirchhoff realisation. Note that the operator described in (3) is not technically the Friedrichs realisation of the Laplacian in the sense of Definition~\ref{def:neumann-friedrichs} due to the presence of the additional Dirichlet conditions at the tips of the teeth.

\begin{proof}[Proof of Theorem~\ref{thm:comb}] 
(1) As mentioned, we will use the compactness criterion of Theorem~\ref{theo:duefel}. Fix $\varepsilon > 0$ and let $f \in H^1(\Graph_\alpha)$ with $\|f\|_{H^1(\mGraph_\alpha)} \leq 1$. Writing $x$ for the point on the shaft identified with $[0,1]$, we have $f(0)=0$. Now, by an application of the fundamental theorem of calculus (and a suitable approximation argument) as well as Cauchy--Schwarz,
\begin{equation}
\label{eq:comb-shaft}
	|f(x)| = \left|\int_0^x f'(t)\,\textrm{d}t\right| \leq \sqrt{x} \|f\|_{H^1(\Graph_\alpha)} = \sqrt{x}.
\end{equation}
An analogous calculation yields that at any point $y \in \Graph_\alpha$ such that $\dist (0,y) \leq \delta$ we have $|f(y)| \leq \sqrt{\delta}$.

Denote by $\me_k$ the $k$-th tooth (counted from $1$ to $0$, i.e. from right to left in the orientation of Figure~\ref{fig:no_compact_embedding}), and note that all points in $\me_k$ are at most at distance $2k^{-\alpha}$ to $0$; hence
\begin{displaymath}
	\|f\|_{L^2(\me_k)}^2 \leq |\me_k|\cdot \max_{y \in \me_k}|f(y)|^2 \leq k^{-\alpha} \cdot 2k^{-\alpha} = \frac{2}{k^{2\alpha}}.
\end{displaymath}
Since $\alpha > \frac{1}{2}$, the sum of upper bounds on all teeth converges.
Hence, summing the respective tails of these series, we obtain that, for sufficiently large $k_0 \in \N$,
\begin{displaymath}
	\sum_{k=k_0}^\infty \|f\|_{L^2(\me_k)}^2 \leq \frac{\varepsilon}{2},
\end{displaymath}
for any $f \in H^1(\Graph_\alpha)$ with $\|f\|_{H^1(\mGraph_\alpha)} \leq 1$. Likewise, \eqref{eq:comb-shaft} yields that on the shaft, possibly for a larger choice of $k_0$,
\begin{displaymath}
	\|f\|_{L^2([0,k_0^{-\alpha}])}^2 \leq \frac{\varepsilon}{2}.
\end{displaymath}
Thus, taking $\Graph_c$ to be the subgraph of the comb cut off at the $k_0$-th tooth, we have shown precisely that \eqref{eq:criterium_compactness} holds for this $\varepsilon$. The compactness of the embedding of $H^1(\Graph_\alpha)$ in $L^2(\Graph_\alpha)$ now follows from Theorem~\ref{theo:duefel}.

(2) We will only give a proof for the critical parameter $\alpha = \frac{1}{2}$ since the case $\alpha \in (0, \frac{1}{2})$ follows by a completely analogous argument. Hence, we consider $\Graph := \Graph_{\frac{1}{2}}$ throughout.

	We construct functions $\phi_n$ as follows:
	For $n \in \N$, let $\phi_n \in C_c(\Graph)$ be
	\[
		\begin{cases}
		&\text{linearly rising from $0$ to $1$
		in $[\frac{1}{\sqrt{2n}}, \frac{1}{\sqrt{n}}]$ on the horizontal shaft $\me$,}
		\\
		&\text{linearly falling from $1$ to $0$
		in $[\frac{1}{\sqrt{n}}, \frac{2}{\sqrt{n}} - \frac{1}{\sqrt{2 n}}]$ on the horizontal shaft,}\\
		&\text{constant on the teeth and and zero everywhere else.}
		\end{cases}
	\]
	It is clear that all $\phi_n$ are in $H^1(\Graph)$, and we calculate
	\[
	\int_\Graph \lvert \phi_n' (x) \rvert^2 \ud x
	=
	2 \left( \frac{1}{\sqrt{n}} - \frac{1}{\sqrt{2n}} \right)^{-1}
	=
	\frac{2}{1 - \sqrt{1/2}} \sqrt{n},
	\]
	and
	\begin{equation}
	\label{eq:L2-norm_of_phi_n}
	\int_\Graph \lvert \phi_n (x) \rvert^2 \ud x
	=
	\frac{2}{3}
	\left( \frac{1}{\sqrt{n}} - \frac{1}{\sqrt{2n}} \right)
	+
	\sum_{k} \frac{1}{\sqrt{k}} \left\lvert \phi_n \left( \frac{1}{\sqrt{k}} \right) \right\rvert^2 ,
	\end{equation}
	where $\phi_n(x)$ denotes the value of $\phi_n$ at the point of the horizontal shaft $\me$ that is identified with $x \in (0,1]$.
	Now, note that $\left\lvert \phi_n \left( \frac{1}{\sqrt{k}} \right) \right\rvert^2$ is certainly greater than $\frac{1}{4}$ if 
	\begin{align*}
	\frac{1}{\sqrt{k}}
	&\in 
	\left[
	\frac{1}{\sqrt{2 n}} + \frac{1}{2} \left(\frac{1}{\sqrt{n}} - \frac{1}{\sqrt{2n}} \right), \frac{1}{\sqrt{n}}
	\right],
	\quad
	\text{equivalently,}
	\quad
	k 
	\in
	\left[
	n, \frac{8}{(\sqrt{2} + 1)^2} n
	\right]
	\supset
	\left[
	n, \frac{4}{3} n
	\right].
	\end{align*}
	Thus, the sum in~\eqref{eq:L2-norm_of_phi_n} contains (for sufficiently large $n$) at least $n/4$ terms, each of which is at least $\frac{1}{4 \sqrt{2 n}}$, and we obtain the lower bound 
	\[
	\int_\Graph
	\lvert \phi_n (x) \rvert^2 \ud x
	\geq
	\frac{n}{16} \frac{1}{\sqrt{2n}} = \frac{\sqrt{n}}{16 \sqrt{2}}.
	\]
	We have found that, for sufficiently large $n$,
	\[
	\frac{\int_\Graph \lvert \phi_n' (x) \rvert^2 \ud x}{\int_\Graph
	\lvert \phi_n (x) \rvert^2 \ud x}
	\leq
	\frac{64 }{ \sqrt{2} - 1}.
	\]
	Since, after passing to a subsequence, we can make the $\phi_n$ have mutually disjoint supports, up to renormalisation they yield an $L^2(\mGraph)$-orthonormal sequence of $H^1(\Graph)$ functions with uniformly bounded $H^1(\Graph)$-norm.
	
(3) This is a direct consequence of Proposition~\ref{prop:length_epsilon_compact}.
\end{proof}

\begin{rem}
	\label{rem:phase_transition}
	Let us compare Theorem~\ref{thm:comb} to results on the Laplacian on \emph{horn shaped domains} $\Omega \subset \R^n$, that are connected domains, bounded in the $x_2, \dots, x_n$-directions, with
	\[
	\lim_{t \to \pm \infty} \operatorname{diam} \{ x \in \Omega \colon x_1 = t \} = 0.
	\] 
	It is known that on any such domain, the spectrum of the \emph{Dirichlet Laplacian} is purely discrete~\cite{Rel48,Ber84}.
	Theorem~\ref{thm:comb}.(3) or more generally the compactness criterion of Proposition~\ref{prop:length_epsilon_compact} seem to be the corresponding analogues on metric graphs.
	
	However, for the \emph{Neumann Laplacian} on domains, the situation is different. 
	Indeed, on connected domains of infinite volume, the essential spectrum of the Neumann Laplacian is always non-empty, see the appendix of~\cite{DavSim92}, which also provides examples of horn-shaped domains of finite volume with non-empty spectrum.
	
	Thus, Theorem~\ref{thm:comb}.(1)--(2), which describes a phase transition from purely discrete to non-empty essential spectrum for the Neumann realisation among metric graphs of infinite volume, seems to present a new, metric-graph specific phenomenon which has no obvious equivalent among domains.
\end{rem}

\begin{rem}
\label{rem:critical-spectrum}
Theorem~\ref{thm:comb}.(2) asserts that the essential spectrum is non-empty but provides no further information on its structure.
However, a closer look at the proof shows that, for $\alpha$ strictly below the critical threshold $\alpha = \frac{1}{2}$, the infimum of the spectrum is zero, since the test functions constructed in the proof will then have Rayleigh quotients converging to zero.
It would be interesting to understand $\inf \sigma (\mathcal{H}_{\mathrm{N}})$ in the critical case $\alpha = \frac{1}{2}$.
\end{rem}

Let us generalise the argument in the proof of Theorem~\ref{thm:comb}.(1) and develop a criterion to determine whether purely discrete spectrum persists in the presence of (finitely many) graph ends of infinite volume.

For this purpose, for any end $\gamma \in \overline{\Graph}$, and $0<r<R$, denote by 
\[
A_{r,R}(\gamma) := \{x \in \Graph: r \leq \dist(x,\gamma) \leq R \}
\]
the (closed) annulus in $\Graph$ about $\gamma$ of inner radius $r$ and outer radius $R$, and by $|A_{r,R}|$ its volume.

\begin{theo}
	\label{theo:discrete_spectrum_annuli}
Let a locally finite, connected metric graph $\Graph$ of finite diameter have only finitely many ends $\gamma_1,\ldots,\gamma_n$ of infinite volume.
Assume that for each $\gamma_i$ there exists some $\alpha_i > 0$ such that $\Graph$ satisfies the following volume growth estimate near $\gamma_i$:
\begin{displaymath}
	|A_{\frac{1}{k+1},\frac{1}{k}}(\gamma_i)| \leq \frac{1}{k^{\alpha_i}}
\end{displaymath}
for all $k\in\N$ sufficiently large. Then $H^1(\Graph)$ embeds compactly in $L^2(\Graph)$.
\end{theo}

The proof of Theorem~\ref{theo:discrete_spectrum_annuli} follows along the lines of the proof of Theorem~\ref{thm:comb}.(2), using the estimate $|f(x)| \leq \sqrt{\dist(x,\gamma_i)}$ for all $x \in \mGraph$ and all ends $\gamma_i$, valid for any $H^1(\Graph)$-function $f$ with norm $1$, to obtain the bound
\begin{displaymath}
	\|f\|_{L^2 \left( A_{\frac{1}{k+1},\frac{1}{k}}(\gamma_i) \right)}^2 \leq \frac{1}{k^{1+\alpha}}
\end{displaymath}
for all $k$ large enough, implying that $\|f\|_{L^2(B_{1/k}(\gamma_i))}^2 \to 0$ as $k \to \infty$. 
In fact, a slight generalisation of Theorem~\ref{theo:discrete_spectrum_annuli} is even possible using the same argument: for each end $\gamma_i$ it suffices that for some sequence $r_k \to 0$ the series $\sum_k r_k |A_{r_{k-1},r_k} (\gamma_i)|$ converges.

\begin{exa}
To show that Theorem~\ref{theo:discrete_spectrum_annuli} also generalises the \emph{result} of Theorem~\ref{thm:comb}.(1), we will apply it to the case of the comb graph $\Graph_\alpha$ (Definition~\ref{def:comb}), for given $\alpha > \frac{1}{2}$. Here, for large $k \in \N$ we can asymptotically bound $|A_{\frac{1}{k+1},\frac{1}{k}}(\gamma_i)|$ from above as follows: the annulus will capture a segment of the horizontal shaft of length $\frac{1}{k} - \frac{1}{k+1} \sim \frac{1}{k^2}$, while it will also include a section of a number of teeth. To estimate this latter number, we need to know asymptotically how many integers $n$ are such that $\frac{1}{n^\alpha} \in [\frac{1}{k+1}, \frac{1}{k}]$, that is, how many integers $n$ are between $k^{1/\alpha}$ and $(k+1)^{1/\alpha}$. Some simple analysis 
shows that this number behaves like $k^{\frac{1}{\alpha}-1}$ as $k \to \infty$. Estimating the length of each tooth from above by $\frac{1}{k}$, we can thus asymptotically bound the size of the annulus from above by
\begin{displaymath}
	|A_{\frac{1}{k+1},\frac{1}{k}}(\gamma_i)| \lesssim k^{\frac{1}{\alpha}-1} \cdot \frac{1}{k} + \frac{1}{k^2} \sim k^{\frac{1}{\alpha}-2},
\end{displaymath}
which is of the form required by Theorem~\ref{theo:discrete_spectrum_annuli} if (and only if) $\alpha > \frac{1}{2}$.
\end{exa}

\subsection{Spectral approximation}\label{sec:spec}

Whenever $\mathcal{H}_{\mathrm{N}}$ and/or $\mathcal{H}_{\mathrm{F}}$ have purely discrete spectrum, or, equivalently, compact resolvent, we will denote by
\begin{displaymath}
	\mu_1 < \mu_2 \leq \mu_3 \leq \ldots \to \infty
\end{displaymath}
the ordered eigenvalues of $\mathcal{H}_{\mathrm{N}}$, repeated according to their multiplicities, and by
\begin{displaymath}
	\lambda_1 < \lambda_2 \leq \lambda_3 \leq \ldots \to \infty
\end{displaymath}
the ordered eigenvalues of $\mathcal{H}_{\mathrm{F}}$. 
We write $\psi_k$ and $\varphi_k$, respectively, for the eigenfunctions associated with $\mu_k$ and $\lambda_k$, respectively, chosen to form orthonormal bases of $L^2(\mGraph)$.

Since $\mGraph$ is connected, and the semigroups generated by both $\mathcal H_{\mathrm{F}}$ and $\mathcal H_{\mathrm{N}}$ are positive and irreducible, standard Perron--Frobenius theory, see e.g.~\cite[Proposition~4.12]{BatKraRha17}, implies that $\mu_1$ and $\lambda_1$ are necessarily simple. 
If $\mGraph$ has finite volume, then $\mu_1 = 0$ with corresponding eigenfunctions being the constant functions, while $\lambda_1 > 0$, as we shall see below. 
In order to emphasise the dependence on $\mGraph$, we will also write $\mu_k (\mGraph)$ and $\lambda_k (\mGraph)$. In the latter case, if there is any danger of ambiguity we will also include the Dirichlet vertex set, $\lambda_k (\mGraph) = \lambda_k (\mGraph, \Dir)$.

The $\mu_n$ and $\lambda_n$ admit min-max and max-min variational characterisations in terms of the forms $\form_{\mathrm{N}}$ and $\form_{\mathrm{F}}$ from \eqref{NeumannForm} and \eqref{FriedrichsForm}, respectively. 
For instance if the volume of $\mGraph$ is finite, we have
\begin{equation}\label{eq:varchar-mu2}
	\mu_2 (\mGraph) = \inf \left\{ \frac{\int_{\mGraph} | f'(x) |^2 \ud x}{\int_{\mGraph} | f(x) |^2 \ud x} : 0 \neq f \in H^1(\mGraph),\, \int_\mGraph f(x)\ud x = 0 \right\},
\end{equation}
\begin{equation}\label{eq:varchar-lambda1}
\begin{split}
	\lambda_1 (\mGraph) 
	&= \inf \left\{ \frac{\int_{\mGraph} | f'(x) |^2 \ud x}{\int_{\mGraph} | f(x) |^2 \ud x} : 0 \neq f \in H^1_0(\mGraph) \right\}
	\end{split}
\end{equation}
and more generally 
\begin{equation}\label{eq:varchar-lambda1-mix}
\begin{split}
	\lambda_1 (\mGraph;\Dir) 
	&= \inf \left\{ \frac{\int_{\mGraph} | f'(x) |^2 \ud x}{\int_{\mGraph} | f(x) |^2 \ud x} : 0 \neq f \in H^1_0(\mGraph;\Dir) \right\}
	\end{split}
\end{equation}
for any non-empty subset  $\Dir \subset \mV \cup \mathfrak C(\Graph)$, see Remark~\ref{rem:mixed-Neumann-Friedrich},
where the infimum in \eqref{eq:varchar-mu2}--\eqref{eq:varchar-lambda1-mix} is attained if and only if $f$ is a corresponding eigenfunction. 
As usual, we call the quotient appearing in the above expressions the \emph{Rayleigh quotient} of $f$.

Let $(\mGraph_n)_{n\in\N}$ be any compact exhaustion of $\mGraph$; we recall (Definition~\ref{DefinitionInducedSubgraph}) that $\partial\Graph_n$ denotes the topological boundary of $\Graph_n$ in $\Graph$. Recalling the notation from Remark \ref{Dirichletknoten} we use $H^1_0 (\mGraph_n;\partial\mGraph_n)$ to denote the subspace of $H^1(\mGraph)$ in which all functions vanish on $\mGraph \setminus \mGraph_n$
and denote by $\lambda_k (\mGraph_n)$ and $\varphi_k (\mGraph_n)$ the eigenvalues and corresponding eigenfunctions of the Laplacian $\mathcal H_{\partial \mGraph_n}$ on $L^2(\mGraph_n)$ (recall Remark \ref{rem:mixed-Neumann-Friedrich}), that is, Dirichlet conditions are satisfied at $\partial\mGraph_n$, and continuity--Kirchhoff conditions are satisfied at all other vertices of $\mGraph_n$. We denote by $\mu_k (\mGraph_n)$ and $\psi_k (\mGraph_n)$ the eigenvalues and eigenfunctions of the Laplacian with continuity--Kirchhoff conditions at all vertices of $\mGraph_n$, respectively; the associated form is $\form_{\mathrm{N}}$ on $H^1 (\mGraph_n)$.

We observe that we may identify any $H^1_0 (\mGraph_n;\partial\mGraph_n)$ with a subspace of $H^1_0 (\mGraph_m;\partial\mGraph_m)$ whenever $m > n$, as well as with a subspace of $H^1_c (\mGraph)$, upon extension by zero of the functions in $H^1_0 (\mGraph_n)$. In fact, it follows directly from the definition of compact approximations and $H^1_c(\mGraph)$ that, with this identification,
\begin{equation}\label{EquationH1c}
	H^1_c (\mGraph) = \bigcup_{n\in\N} H^1_0 (\mGraph_n;\partial\mGraph_n).
\end{equation}

The following approximation result will be used repeatedly in the following sections.

\begin{lemma}\label{TheoremEigenvalueApproximation}
Let $\mGraph$ be a locally finite, connected metric graph of finite volume and let $(\mGraph_n)_{n\in\N}$ be any compact exhaustion of $\mGraph$. Then the following assertions hold.
\begin{enumerate}
\item For all $k \in \N$ and all $n \in \N$ we have $\lambda_k (\mGraph_n) \geq \lambda_k (\mGraph)$, and $\lambda_k (\mGraph_n) \to \lambda_k (\mGraph)$ as $n \to \infty$.
\item For all $k \in \N$ we also have
\begin{equation}\label{EquationNeumannLimsup}
	\mu_k (\mGraph) \geq \limsup_{n\to \infty} \mu_k (\mGraph_n).
\end{equation}
\item If in addition $\mGraph$ has finite Betti number, then for every $k \in \N$ we have $\mu_k (\mGraph_n) \geq \mu_k (\mGraph)$ for all $n \in \N$ sufficiently large.
\end{enumerate}
\end{lemma}

In particular, if $\mGraph$ has finite Betti number, then $\mu_k (\mGraph_n) \to \mu_k (\mGraph)$ as $n \to \infty$, for any compact exhaustion.

The proof of part (1) is closely related to the notion of Mosco convergence~\cite{Mos67}, part (2) involves an elementary argument using the restriction of the eigenfunctions on $\mGraph$ to $\mGraph_n$, and part (3) follows from a surgery-type argument. In the general case it does not seem clear whether there is actually convergence $\mu_k (\mGraph_n) \to \mu_k (\mGraph)$ for $k \geq 2$ (the case $k=1$ is trivial).

\begin{proof}
(1) The inequality $\lambda_k (\mGraph_n) \geq \lambda_k (\mGraph)$ for all $k,n \in \N$ is an immediate consequence of the identification of every element of $H^1_0 (\mGraph_n)$ with an element of $H^1_c (\mGraph)$, and hence of $H^1_0 (\mGraph)$, via extension by zero, together with the min-max principle for $\lambda_k$.

On the other hand, given any $u \in H^1_0 (\mGraph)$, by \eqref{EquationH1c} and the definition of $H^1_0 (\mGraph)$ as the closure in $H^1$ of $H^1_c (\mGraph)$, there exists a sequence of functions $u_n \in H^1_0 (\mGraph_n)$ such that $u_n \to u$ in $H^1_0 (\mGraph)$. Fix $k \geq 1$, write $\varphi_j := \varphi_j (\mGraph) \in H^1_0 (\mGraph)$ for the normalised eigenfunctions for $\lambda_j (\mGraph)$, $j=1,\ldots,k$, which will be our $u$ and for each $j$ choose $u_{j,n} \in H^1_0 (\mGraph_n)$ such that $u_{j,n} \to \varphi_j$ in $H^1_0 (\mGraph)$ (and hence also in $C(\overline{\mGraph})$). Note that, by the dominated convergence theorem (using, e.g., $2|\varphi_i\varphi_j| \in L^\infty(\mGraph)$ as a dominating function, cf. \cite[Lemma 3.2]{KosMugNic22}), for all $i \neq j$,
\begin{displaymath}
	\int_{\mGraph} u_{i,n}(x)u_{j,n}(x)\,\ud x \longrightarrow \int_{\mGraph} \varphi_i(x)\varphi_j(x)\,\ud x = 0.
\end{displaymath}
If we now consider the following renormalised test functions, mutually orthogonal in $L^2(\mGraph)$,
\begin{displaymath}
	\tilde u_{j,n} := u_{j,n} - \sum_{i=1}^{j-1} \langle u_{i,n}, u_{j,n}\rangle_{L^2(\mGraph)} \varphi_{i,n},
\end{displaymath}
then the above convergence results imply that $\tilde u_{j,n}$ is nonzero for $n$ sufficiently large, and an elementary computation yields
\begin{equation}\label{EquationMoscoStyle}
	\lambda_k (\mGraph_n) \leq \max_{j=1,\ldots,k} \frac{\int_{\mGraph} |\tilde u_{j,n}'(x)|^2\,\ud x}{\int_{\mGraph} |\tilde u_{j,n}(x)|^2\,\ud x} \longrightarrow
	\max_{j=1,\ldots,k} \frac{\int_\mGraph |\varphi_j'(x)|^2\,\ud x}{\int_\mGraph |\varphi_j(x)|^2\,\ud x} = \lambda_k (\mGraph),
\end{equation}
as $n \to \infty$, where for the first inequality we have used the min-max characterisation of $\lambda_k (\mGraph_n)$. Hence $\limsup_{n\to \infty} \lambda_k (\mGraph_n) \leq \lambda_k (\mGraph)$.

(2) We use the respective restrictions of $\psi_j := \psi_j (\mGraph) \in H^1(\mGraph)$ to $H^1 (\mGraph_n)$, $j=1,\ldots,k$, as test functions on $\mGraph_n$. Firstly, since the measure of $\mGraph \setminus \mGraph_n$ tends to zero, by the monotone convergence theorem,
\begin{displaymath}
	\|\psi_j\|_{L^2(\mGraph_n)} \to \|\psi_j\|_{L^2(\mGraph)}, \qquad \|\psi_j'\|_{L^2(\mGraph_n)} \to \|\psi_j'\|_{L^2(\mGraph)}
\end{displaymath}
for all $j=1,\ldots,k$; moreover, a further application of the dominated convergence theorem (using, e.g., $|\psi_i\psi_j| \in L^\infty(\mGraph)$ as a dominating function) shows that
\begin{displaymath}
		\int_{\mGraph_n} \psi_i(x)\psi_j(x)\,\ud x \to 0.
\end{displaymath}
On $\mGraph_n$ we thus consider the $k$-dimensional space spanned by the orthogonal functions
\begin{displaymath}
	\tilde\psi_{j,n} := \psi_j - \sum_{i=1}^{j-1} \langle \psi_i,\psi_j\rangle_{L^2(\mGraph_n)} \psi_i.
\end{displaymath}
An argument entirely analogous to the one in part (1) now shows that, for each $j=1,\ldots,k$, $\tilde\psi_{j,n}$ is nonzero for $n\in\N$ sufficiently large, and \eqref{EquationMoscoStyle} holds with $\mu_k$ in place of $\lambda_k$, and $\tilde\psi_{j,n}$, $\tilde\psi_j$ in place of $\tilde\varphi_{j,n}$, $\tilde\varphi_j$. This proves the claim.

(3) Here we will use in an essential way that, for all $n$ sufficiently large, $\mGraph \setminus \mGraph_n$ consists of a disjoint union of a finite number of pairwise disjoint trees, each of which is attached to $\mGraph_n$ at a single vertex, cf.\ Remark~\ref{RemarkAlmostTree}. 
Fix such an $n$ and denote by $\mathcal{T}_{1}, \ldots, \mathcal{T}_{j}$ these trees, and by $\mv_1, \ldots, \mv_{j}$ the corresponding vertices of attachment.

By induction on the trees $\mathcal{T}_{1}, \ldots, \mathcal{T}_{j}$, it suffices to show that $\mu_k (\mGraph) \leq \mu_k (\mGraph \setminus \mathcal{T}_{1})$ for all $k \in \N$. But this, in turn, follows from an argument completely analogous to the proof of \cite[Theorem~3.10(1)]{BerKenKur19} (with $r=1$): we first observe that clearly
\begin{equation}
\label{eq:bkkm-part1}
	\mu_{k+1} (\mathcal{T}_1 \dot{\cup} (\mGraph \setminus \mathcal{T}_1)) \leq \mu_k (\mGraph \setminus \mathcal{T}_1)
\end{equation}
for all $k \in \N$, since the spectrum of the disjoint union of $\mathcal{T}_1$ and $\mGraph \setminus \mathcal{T}_1$ is equal to the union of their spectra (counting multiplicities), and $\mu_1 (\mathcal{T}_1) = 0 \leq \mu_k (\mGraph \setminus \mathcal{T}_1)$ for all $k \in \N$. Next observe that $\mGraph$ is formed from $\mathcal{T}_1 \dot{\cup} (\mGraph \setminus \mathcal{T}_1)$ by gluing the two graphs together at the vertex $\mv_1$; at the level of $H^1$-spaces, we have that $H^1(\mGraph)$ may be identified with the codimension one subspace of $H^1(\mathcal{T}_1 \dot{\cup} (\mGraph \setminus \mathcal{T}_1))$ consisting of those functions whose values satisfy an additional continuity condition at $\mv_1$ (or rather its preimage vertices in $\mathcal{T}_1$ and $\mGraph \setminus \mathcal{T}_1$ before gluing). The proof of \cite[Theorem~3.4]{BerKenKur19} may be repeated verbatim to give
\begin{displaymath}
	\mu_k (\mGraph) \leq \mu_{k+1} (\mathcal{T}_1 \dot{\cup} (\mGraph \setminus \mathcal{T}_1)).
\end{displaymath}
Combining this with \eqref{eq:bkkm-part1} yields the conclusion.
\end{proof}

\section{Lower bounds on the lowest positive eigenvalue}\label{sec:symmetrisation}

This section contains lower bounds on the lowest positive eigenvalue in terms of geometric properties of the underlying metric graph.

\subsection{Symmetrisation techniques for isoperimetric inequalities}\label{subsec:symmetrisation}

We start with isoperimetric inequalities for the lowest positive eigenvalue of both Friedrichs and Neumann realisations of the Laplacian on an infinite metric graph $\Graph$. 
Our estimates extend several bounds only known in the case where $\Graph$ is compact.
While the inequalities themselves could also be obtained via an approximation argument using surgery principles generalised to infinite graphs, a suitable compact exhaustion, and Lemma~\ref{TheoremEigenvalueApproximation}, we wish to go further and also identify the cases of equality.
This requires somewhat more intricate proofs based on symmetrisation and the coarea formula, techniques going back to Friedlander~\cite{Fri05} and Band--L\'evy~\cite{BanLev17,BerKenKur17}. 
Indeed, the following extends \cite[Theorem~1]{Fri05} to the case of graphs with nonempty set of ends; its proof will be presented at the end of this subsection, after collecting a number of auxiliary results.

\begin{theo}
\label{FriedlanderTh1}
Let $ \mGraph $ be a locally finite, connected metric graph with finite volume $|\mGraph| >0 $. Then for the $ k$-th eigenvalue $ \mu_k(\mGraph) $, $ k \geq 2 $, of the Neumann realisation $\mathcal{H}_{\mathrm{N}}$ on $\mGraph$,
\begin{equation}\label{FriedlanderTh1Gl1}
	\mu_k{( \mGraph)} \geq \frac{\pi^2 k^2}{ 4 |\mGraph|^2}.
\end{equation}
There is equality if and only $\mGraph$ is a star consisting of $ k$  
path subgraphs of length $|\Graph|/k$ each, glued together at a common vertex.
\end{theo}


For the case of the Friedrichs realisation we have the equivalent of the Dirichlet version of the theorems of Nicaise and Friedlander, see~\cite[Théorème 3.2]{Nic87} and~\cite[Lemma~3]{Fri05}, respectively. 
Here and in the sequel, let $\Dir \subset \overline{\mGraph} \subset \mV \cap \mathfrak{C}(\mGraph)$, and in place of the Friedrichs realisation $\mathcal{H}_{\mathrm{F}}$ consider the more general $\mathcal{H}_{\Dir}$ of Remark~\ref{rem:mixed-Neumann-Friedrich}imposing Dirichlet conditions on an (essentially arbitrary) subset $\Dir$ of the union of the set of all vertices and the set of all ends of $\mGraph$.

\begin{theo}
\label{estimate-friedr-nic}
Let $ \mGraph $ be a locally finite, connected metric graph with finite volume $|\mGraph|>0 $ and assume $\Dir \neq \emptyset$. Then the lowest eigenvalue $\lambda_1 (\mGraph,\Dir)$ of the mixed Friedrichs--Neumann realisation $\mathcal{H}_{\Dir}$ on $\mGraph$ satisfies
\begin{equation}\label{eq:nicaise-dir} 
	\lambda_1{( \mGraph,\Dir)} \geq \frac{\pi^2}{ 4|\mGraph|^2}.
\end{equation}
Equality in \eqref{eq:nicaise-dir} is attained if and only if $\mGraph$ is isometrically isomorphic to an interval of length $|\mGraph|>0 $ with mixed Dirichlet--Neumann conditions.
\end{theo}

In the above theorem, the case $\Dir = \mathfrak{C}(\mGraph)$ describes the Friedrichs realisation $\mathcal{H}_{\mathrm{F}}$.

Finally, for doubly path connected graphs, we have the following improvements of \eqref{FriedlanderTh1Gl1} (for the first nontrivial eigenvalue $\mu_2$) and \eqref{eq:nicaise-dir}. 
Their counterparts for finite metric graphs are \cite[Theorem~2.1]{BanLev17} and~\cite[Lemma~4.3]{BerKenKur17}, respectively.

\begin{theo}
\label{TheoremBandLevy}
	Let $\mGraph$ be a locally finite, doubly connected metric graph with finite volume $|\mGraph| >0 $. 
	Then the first nontrivial eigenvalue $\mu_2 (\mGraph)$ of the Neumann realisation $\mathcal{H}_{\mathrm{N}}$ on $\mGraph$ satisfies
	\begin{equation}\label{eq:bandlevy}
		\mu_2 (\mGraph) \geq \frac{4\pi^2}{|\mGraph|^2}.
	\end{equation}
Equality in \eqref{eq:bandlevy} holds if and only if $\mathcal{G}$ is a 
	symmetric necklace.
\end{theo}

Recall that a \emph{symmetric necklace}, see~\cite[Example~1.7]{BanLev17}, is a metric graph obtained concatenating (a finite or countably infinite number of) pumpkin graphs, each on two equally long edges. We emphasise that in the infinite case it is allowed to have two ends.

\begin{theo}
\label{TheoremBKKM1}
Let $\mGraph$ be a locally finite, connected metric graph with finite volume $|\mGraph| >0 $.  Take $\emptyset \neq \Dir \subset \mV \cap \mathfrak{C}_\Graph$ to contain all degree one vertices of $\Graph$ and suppose that for all $x \in \mGraph$ there are at least two edge-disjoint paths connecting $x$ and $\Dir$.

Then the first nontrivial eigenvalue $\lambda_1 (\mGraph, \Dir)$ of $\mathcal{H}_{\Dir}$ on $\mGraph$ satisfies
\begin{equation}\label{eq:bkkm1}
	\lambda_1 (\mGraph, \Dir) \geq \frac{\pi^2}{|\mGraph|^2}.
\end{equation}
Equality in~\eqref{eq:bkkm1} implies that $\mGraph$ is a symmetric necklace such that either (1) $\Dir$ consists of one end, equipped with the Friedrichs realisation, or (2) the necklace is finite with a single Dirichlet condition imposed at one of its extremities, or (3) $\Dir$ consists of one end, equipped with the Neumann realisation, but there is a single Dirichlet condition at the other extremity of the necklace, or (4) the necklace has two ends, with a Neumann realisation at one and a Dirichlet realisation at the other.
\end{theo}

{See also Figure~\ref{fig:caterpillar} for an illustration. It is convenient to think of the two Dirichlet vertices as being glued together (in opposition to our usual approach) due to parallels with the minimising graphs in Theorem~\ref{TheoremBandLevy}.}

\begin{figure}

	\begin{tikzpicture}
		\begin{scope}[yshift = 2cm, xshift= -0.5cm] 
				\draw[fill] (-1.5,0) circle (3pt);
				\draw[thick] (-0.8,0) circle (.7cm);
				\draw[fill] (-0.1,0) circle (3pt);
				\draw[thick] (0.2,0) circle (.3cm);
				\draw[fill] (0.5,0) circle (3pt);
				\draw[thick] (1,0) circle (.5cm);
				\draw[fill] (1.5,0) circle (3pt);
				\draw[thick] (1.95,0) circle (.45cm);
				\draw[fill] (2.4,0) circle (3pt);
				\draw[thick] (2.8,0) circle (.4cm);
				\draw[fill] (3.2,0) circle (2.5pt);
				\draw[thick] (3.5,0) circle (.3cm);
				\draw[fill] (3.8,0) circle (2.5pt);
				\draw[thick] (4,0) circle (.2cm);
				\draw[fill] (4.2,0) circle (2pt);
				\draw[thick] (4.35,0) circle (.15cm);
				\draw[fill] (4.5,0) circle (2pt);
				\draw[thick] (4.6,0) circle (.1cm);
				\draw[fill] (4.7,0) circle (1pt);
				\draw[thick] (4.75,0) circle (.05cm);
				\draw[fill] (4.8,0) circle (1pt);
				\draw[thick] (4.825,0) circle (.025cm);
				
				\draw[thick] (4.875,0) ellipse (.025cm and .02cm);
				\draw[thick] (4.925,0) ellipse (.025cm and .015cm);
				\draw[thick] (4.975,0) ellipse (.025cm and .01cm);						
				\draw[thick, fill = gray!20] (5,0) circle (3pt);
		
		\end{scope}

		\begin{scope}
			\foreach \x in {-1,0,3}
				{
				\draw[fill] ({\x - .5},0) circle (3pt);
				\draw[fill] ({\x + .5},0) circle (3pt);
				\draw[thick] (\x,0) circle (.5cm);
				}
			\draw[thick] (0.8,0) circle (.3cm);
			\draw[fill] (1.1,0) circle (3pt);
			\draw[thick] (1.8,0) circle (.7cm);
			\draw[thick, fill = gray!20] (-1.5,0) circle (3pt);
		\end{scope}

		\begin{scope}[yshift = -2cm]
			\foreach \x in {-1}
				{
				\draw[fill] ({\x - .5},0) circle (3pt);
				\draw[fill] ({\x + .5},0) circle (3pt);
				\draw[thick] (\x,0) circle (.5cm);
				}
				\draw[thick] (0,0) circle (.5cm);
				\draw[fill] (0.5,0) circle (3pt);
				\draw[thick] (1.1,0) circle (.6cm);
				\draw[fill] (1.7,0) circle (2.5pt);
				\draw[thick] (2,0) circle (.3cm);
				\draw[fill] (2.3,0) circle (2.5pt);
				\draw[thick] (2.5,0) circle (.2cm);
				\draw[fill] (2.7,0) circle (2pt);
				\draw[thick] (2.85,0) circle (.15cm);
				\draw[fill] (3,0) circle (2pt);
				\draw[thick] (3.1,0) circle (.1cm);
				\draw[fill] (3.2,0) circle (1pt);
				\draw[thick] (3.25,0) circle (.05cm);
				\draw[fill] (3.3,0) circle (1pt);
				\draw[thick] (3.325,0) circle (.025cm);
				
				\draw[thick] (3.375,0) ellipse (.025cm and .02cm);
				\draw[thick] (3.425,0) ellipse (.025cm and .015cm);
				\draw[thick] (3.475,0) ellipse (.025cm and .01cm);	
			\draw[thick, fill = gray!20] (-1.5,0) circle (3pt);
		\end{scope}

		\begin{scope}[yshift = -4cm, xshift = 0.5cm]
				\draw[thick] (0,0) circle (.5cm);
				\draw[fill] (0.5,0) circle (3pt);
				\draw[thick] (1.1,0) circle (.6cm);
				\draw[fill] (1.7,0) circle (2.5pt);
				\draw[thick] (2,0) circle (.3cm);
				\draw[fill] (2.3,0) circle (2.5pt);
				\draw[thick] (2.5,0) circle (.2cm);
				\draw[fill] (2.7,0) circle (2pt);
				\draw[thick] (2.85,0) circle (.15cm);
				\draw[fill] (3,0) circle (2pt);
				\draw[thick] (3.1,0) circle (.1cm);
				\draw[fill] (3.2,0) circle (1pt);
				\draw[thick] (3.25,0) circle (.05cm);
				\draw[fill] (3.3,0) circle (1pt);
				\draw[thick] (3.325,0) circle (.025cm);
				
				\draw[thick] (3.375,0) ellipse (.025cm and .02cm);
				\draw[thick] (3.425,0) ellipse (.025cm and .015cm);
				\draw[thick] (3.475,0) ellipse (.025cm and .01cm);

				\draw[thick] (-0.75,0) circle(.25cm);
				\draw[fill] (-0.5,0) circle (3pt);
				\draw[thick] (-1.25,0) circle(.25cm);
				\draw[fill] (-1,0) circle (3pt);
				\draw[fill] (-1.5,0) circle (2pt);
				\draw[thick] (-1.6,0) circle (.1cm);
				\draw[fill] (-1.7,0) circle (1pt);
				\draw[thick] (-1.75,0) circle (.05cm);
				\draw[fill] (-1.8,0) circle (1pt);
				\draw[thick] (-1.825,0) circle (.025cm);	

				\draw[thick] (-1.85,0) ellipse (.025cm and .02cm);
				\draw[thick] (-1.9,0) ellipse (.025cm and .015cm);
				\draw[thick] (-1.95,0) ellipse (.025cm and .01cm);	
			\draw[thick, fill = gray!20] (-1.975,0) circle (3pt);
		\end{scope}
	\end{tikzpicture}

	\caption{The four cases where equality in~\eqref{eq:bkkm1} holds.  Dirichlet conditions (at vertices or ends) are coloured light grey, standard/Neumann conditions are depicted as filled black circles.
	(1) In the top graph the Friedrichs realisation leads to a formal Dirichlet condition at the end on the right side of the infinite necklace; (2) the second graph is a compact necklace with a Dirichlet condition at one end and a standard condition at the other; (3) in the third row there is a Dirichlet vertex of degree two at the finite end of the necklace and a Neumann condition at the other end; (4) the bottom graph has a Dirichlet condition at the end on the left and a Neumann condition at the end on the right.}
	\label{fig:caterpillar}
\end{figure}

The topological assumption on $\mGraph$ in Theorem~\ref{TheoremBKKM1} is equivalent to requiring that the metric space obtained from $\mGraph$ upon identifying all points in $\Dir$, must be doubly connected.

The proofs will make use of the following technical result, which is a consequence of the coarea formula (cf.\ \cite[Proof of Lemma~3]{Fri05}, \cite[Proof of Theorem~2.1]{BanLev17} and \cite[Proof of Theorem~3.4]{BerKenKur17}).
Notationally, given a graph $\mGraph$ and a measurable function $f: \mGraph \to \R$, we denote by 
\[
S_t = S_t(f) := \{ x \in \mGraph : f(x) = t\} \subset \mGraph
\] 
the level surface of $f$ and by 
\[
m_f (t) := | \{x \in \mGraph: f(x) < t \}|
\]
the Lebesgue measure of its sublevel set, for any $t \in \R$.

\begin{lemma}\label{LemmaCoarea}
Let $\mGraph$ be a locally finite metric graph of finite volume, and let $f \in H^1(\mGraph) \cap \bigoplus_{\me\in\mE}C^1([0,\ell_\me])$. Then, with the notation just introduced, $m_f: \R \to \R$ is absolutely continuous; moreover, for almost every $t \in \R$, $S_t$ has finite cardinality and
\begin{equation}\label{eq:mfrate}
	m_f'(t) = \sum_{x \in S_t} \frac{1}{|f'(x)|};
\end{equation}
finally,
\begin{equation}\label{eq:ourcoarea}
	\int_\mGraph |f'(x)|^2\,\ud x = \int_\R \sum_{x \in S_t} |f'(x)|\, \mathrm{d}t.
\end{equation}
\end{lemma}

Since the key ingredient of the proof is the coarea formula, which to the best of our knowledge has not been explicitly formulated for infinite metric graphs, we first give a version of the latter; we remark that, at least in the case of compact graphs, a generalisation of the following result to functions of bounded variations was recently given  in~\cite{Maz23}.

\begin{lemma}[Coarea formula for infinite graphs]
\label{lem:true-coarea}
Let $\mGraph$ be any metric graph on a countable edge set, let $f \in C(\mGraph) \cap \bigoplus_{\me\in\mE}C^1([0,\ell_\me])$ and let $\varphi \in L^1_{\textrm{loc}} (\mGraph)$ be nonnegative.
Then, with the notation introduced above,
\begin{equation}
\label{eq:truecoarea}
	\int_\mGraph \varphi(x) |f'(x)|\,\ud x = \int_{\R} \sum_{x \in S_t} \varphi(x)\,\ud t.
\end{equation}
\end{lemma}

\begin{proof}

Fix $f \in C(\mGraph) \cap \bigoplus_{\me\in\mE}C^1([0,\ell_\me])$. Call a value $t \in \R$ exceptional if $S_t \cap \mV \neq \emptyset$ or $f'(x)=0$ for some $x \in S_t$. Then by Sard's theorem~\cite{Sar42} the set of exceptional values of $f|_\me$ is countable for all $\me \in \mE$ (and hence the set of exceptional values of $f$ on $\mGraph$ is also countable); and for fixed $\varphi$, \eqref{eq:truecoarea} holds on each edge of $\mGraph$. Thus it also holds on any finite union of edges; taking a sequence of finite subgraphs which exhaust $\mGraph$ and passing to the limit using the monotone convergence theorem yields \eqref{eq:truecoarea} for $f$ and $\varphi$ on $\mGraph$.
\end{proof}

\begin{proof}[Proof of Lemma~\ref{LemmaCoarea}]
Take $f$ as in the statement of the lemma and apply \eqref{eq:truecoarea} to $f$ (with $\varphi = f$ as well); this immediately yields \eqref{eq:ourcoarea}. Since under our current assumptions $\mGraph$ has finite volume, we also have $f \in W^{1,1} (\mGraph)$. Repeating the argument but taking $\varphi = 1$ shows that
\begin{displaymath}
	\int_\R \sum_{x \in S_t} 1 \,dt = \int_\mGraph |f'(x)|\,\ud x < \infty,
\end{displaymath}
which in particular implies that $S_t$ is finite for almost all $t \in \R$. We now restrict to the set of regular values $t \in \mathfrak{R}(f)$ for which additionally $S_t$ is finite; this possibly smaller set still has full measure in $\R$. Now by construction, for all such $t$ we have
\begin{displaymath}
	S_t = \bigcup_{i \in I} S_t \cap \me_i, \qquad m_f (t) = \sum_{i \in I} m_{f_i} := \sum_{i \in I} |\{ x \in \me_i : f(x) < t \}|,
\end{displaymath}
and
\begin{displaymath}
	m_f' (t) = \sum_{i \in I} m_{f_i}'(t) = \sum_{i \in I} \sum_{x \in S_t \cap \me_i} \frac{1}{|f'(x)|} = \sum_{x \in S_t} \frac{1}{|f'(x)|},
\end{displaymath}
where the finiteness of $S_t$ implies the finiteness of all the sums involved, and hence the validity of the identities, as well as the absolute continuity of $m_f$. This completes the proof.
\end{proof}

With the help of Lemma~\ref{LemmaCoarea}, the key symmetrisation argument used in the compact case can be generalised directly to infinite graphs. For this we need to introduce some more notation: given a graph $\mGraph$ of volume $|\mGraph|>0 $ and a function $f \in H^1 (\mGraph) \hookrightarrow C(\overline{\mGraph}) \hookrightarrow L^\infty (\mGraph)$, we define its symmetrisation (more precisely: its decreasing rearrangement) $f^\ast \in C([0,L])$ via the level set property
\begin{displaymath}
	|\{ x \in (0,L) : f^\ast (x) < t \}| = m_f (t) \qquad \text{for all } t \in \R
\end{displaymath}
(and extension by continuity to $x=0,L$). Taking $f \in H^1 (\mGraph)$ to be fixed, we also set $n(t) := \# S_t = \#\{ x \in \mGraph: f(x) = t\}$, which is a nonnegative integer for almost all $t \in \R$ by Lemma~\ref{LemmaCoarea}; and in fact is (again, for almost all $t \in \R$) at least 2 under the assumptions of Theorems~\ref{TheoremBandLevy} and~\ref{TheoremBKKM1}, by an argument essentially based on Menger's Theorem as in~\cite{BanLev17,BerKenKur17}. The following result is now standard; its proof, using the properties established in Lemma~\ref{LemmaCoarea}, follows the proof of, e.g., \cite[Theorem~2.1]{BanLev17} essentially verbatim (see in particular Eq.\ (3.13) there). We therefore omit it.

\begin{lemma}\label{LemmaSymmetrisation}
Let $\mGraph$ be a locally finite, connected metric graph of volume $|\mGraph| >0 $, let $f \in H^1(\mGraph) \cap \bigoplus_{\me\in\mE}C^1([0,\ell_\me])$, with minimum $m$ and maximum $M \geq m$ in $\overline{\mGraph}$, respectively, and let $f^\ast$ be its increasing rearrangement, as just described. Then $f^\ast \in H^1 (0,L)$,
\begin{equation}\label{eq:cavalieri}
	\int_\mGraph |f(x)|^2\,\ud x = \int_0^L |f^\ast (x)|^2\,\ud x,
\end{equation}
and
\begin{equation}\label{eq:symmetrisation}
	\int_\mGraph |f'(x)|^2\,\ud x \geq \essinf\limits_{t \in [m,M]} n(t)^2 \int_0^L |(f^\ast)'(x)|^2\,\ud x.
\end{equation}
Equality in \eqref{eq:symmetrisation} implies that $f'(x)$ takes on a common value at all $x\in S_t$, for almost all $t \in [m,M]$.
\end{lemma}

As a consequence we obtain the functional inequality of Friedlander \cite[Lemma~3]{Fri05} for infinite graphs. We set this up to work simultaneously for both $H^1$-functions vanishing at at least one point in $\mGraph$ and functions in $H^1_0 (\mGraph)$; the result for the latter will immediately imply Theorem~\ref{estimate-friedr-nic}.
Recall that $ H_0^1(\mGraph;\{y\}) \subset  H^1(\mGraph) $ denotes the subspace of $H^1$-functions that vanish at $ y $, whether $y$ is a vertex or an end.

\begin{lemma}\label{UnendlichFried}
Let $ \mGraph $ be a locally finite, connected metric graph with finite volume $|\mGraph| >0 $ and let $ y \in \overline{\mGraph} $.  
Then
\begin{equation}\label{friedlanderlemma1}
	\int_{\mGraph} |f'(x)|^2\, \ud x \geq  \frac{\pi^2}{4 |\mGraph|^2}\int_{\mGraph} |f(x)|^2 \,\ud x
\end{equation}
for all $ f \in H_0^1(\mGraph;\{y\}) \cap \bigoplus_{\me\in\mE}C^1([0,\ell_\me]) $. 
In either case, equality in \eqref{friedlanderlemma1} can occur for a nonzero function $f$ if and only if 
$ \mGraph $ is isometrically isomorphic to an interval, with $y$ being one of its endpoints,
and $ f $ is proportional to $ \sin(\pi s/2 |\mGraph|) $, where $ s $ is the distance to $ y $.
\end{lemma}

\begin{proof}
In either case, by replacing $f$ by $|f|$ we may assume that $f$ is nonnegative. 
If $y$ is a vertex, then $f \in H_0^1(\mGraph;\{y\})$; since $f$ is continuous and has a zero in $\mGraph$, $f^\ast (0) = 0$. It follows from Lemma~\ref{LemmaSymmetrisation}, the fact that $n(t) \geq 1$ for all $t \in (m,M)$ by continuity of $f$, and the usual one-dimensional inequality (cf.\ \cite[Eq.\ (2.7)]{Fri05}) that
\begin{equation}\label{eq:friedlander-intermediate}
	\int_\mGraph |f'(x)|^2\,\ud x \geq \int_0^L |(f^\ast)'(x)|^2\,\ud x \geq \frac{\pi^2}{4|\mGraph|^2}\int_0^L |f^\ast(x)|^2 \,\ud x = \frac{\pi^2}{4|\mGraph|^2}\int_{\mGraph} |f(x)|^2 \,\ud x.
\end{equation}
If $y$ is an end, then $f(x) \to 0$ at the end $y$, necessarily $\essinf_{x \in \mGraph} f(x) = 0$ and so $f^\ast (0) = 0$, from which \eqref{eq:friedlander-intermediate} follows.

If there is equality in \eqref{friedlanderlemma1}, and hence in every step of \eqref{eq:friedlander-intermediate}, then $n(t)=1$ for almost all $t \in (m,M) = (0,M)$ and hence for \emph{all} $t \in (0,M)$; moreover, since up to scalar multiples and reflections $f^\ast (x) = \sin(\pi x/2 L)$, $f$ can only vanish on a set of measure zero, and hence at most at a single point. That is, $f$ can take on any value between $0$ and its maximum on $\overline{\mGraph}$ only once. A standard continuity argument shows that $\mGraph$ must be a path graph, with $f$ vanishing at exactly one of its endpoints. The fact that $f$, considered as a function on $[0,L]$, minimises the Rayleigh quotient of the second derivative on $(0,L)$ with a Dirichlet condition at $0$, means that it must be proportional to the first eigenfunction $\sin (\frac{\pi}{2|\Graph|}\, \cdot \,)$.
\end{proof}

We can now give the proof of Theorems~\ref{estimate-friedr-nic},~\ref{TheoremBandLevy}  and~\ref{TheoremBKKM1}; the proof of Theorem~\ref{FriedlanderTh1Gl1} requires a further intermediate result and will be given afterwards.

\begin{proof}[Proof of Theorem~\ref{estimate-friedr-nic}]
This now follows immediately from Lemma~\ref{UnendlichFried} applied to the first eigenfunction of $\mathcal{H}_{\Dir}$, together with the variational characterisation \eqref{eq:varchar-lambda1-mix} of $\lambda_1 (\mGraph;\Dir)$.
\end{proof}

\begin{proof}[Proof of Theorem~\ref{TheoremBandLevy}]
Let $\psi_2 \in H^1 (\mGraph)$ be any eigenfunction for $\mu_2 (\mGraph)$, then it necessarily changes sign on $\mGraph$, being orthogonal to the constant functions. We next claim that, if $m = \min_{x \in \overline{\mGraph}} \psi_2 (x)$ and $M = \max_{x \in \overline{\mGraph}} \psi_2 (x)$, then $n(t) = 2$ for almost all $t \in (m,M)$: indeed, choose $x_n,y_n \in \mGraph$ such that $\psi_2(x_n)=:m_n \to m$ and $\psi_2 (y_n)=:M_n \to M$. Fix $n \in \N$. Then the assumption that $\mGraph$ is doubly connected means that there are two paths $P_1$ and $P_2$ connecting $x_n$ and $y_n$, intersecting at a null set of vertices. Since $\psi_2$ is certainly continuous, it takes on every value between $m_n$ and $M_n$ at least once on each paths; hence $n(t) = 2$ for all $t \in (m_n,M_n)$ except possibly at the finite set of values $\psi_2$ takes at the vertices. Now let $n \to \infty$; since the countable union of finite sets is certainly a null set, we have proved the claim.

We next let $\mGraph_0$ be any \emph{nodal domain} of $\psi_2$, that is, the closure of a connected component of $\{x \in \mGraph : \psi_2 \neq 0 \}$. Let $\eta \in H^1 (\mGraph)$ be the restriction of $\psi_2$ to $\mGraph_0$, extended by zero on $\mGraph \setminus \mGraph_0$; then taking $\eta$ as a test function in the weak formulation of the eigenvalue problem
\begin{displaymath}
	\int_\mGraph \psi_2'(x) v'(x) \,\ud x = \mu_2 (\mGraph) \int_\mGraph \psi_2 (x)v(x)\,\ud x \qquad \text{for all } v \in H^1 (\mGraph)
\end{displaymath}
leads to
\begin{equation}\label{eq:nodaldomainproperty}
	\mu_2 (\mGraph) = \frac{\int_{\mGraph} |\eta'(x)|^2\,\ud x}{\int_\mGraph |\eta (x)|^2\,\ud x} = \frac{\int_{\mGraph_0} |\psi_2'(x)|^2\,\ud x}{\int_{\mGraph_0} |\psi_2(x)|^2\,\ud x},
\end{equation}
just as holds in the compact case. Using that $\psi_2(x) = 0$ for at least one $x \in \mGraph_0$ and that $n(t)=2$ for a.e.\ $t\in [0,\max \psi_2]$, Lemma~\ref{LemmaSymmetrisation} together with an argument analogous to \eqref{eq:friedlander-intermediate} implies that
\begin{displaymath}
	\mu_2 (\mGraph) \geq \frac{\pi^2}{|\mGraph_0|^2}.
\end{displaymath}
Now this holds for every nodal domain $\mGraph_0$. Since there are at least two, as noted above, at least one has volume at most $|\mGraph|/2$, which leads to the estimate $\mu_2 (\mGraph) \geq 4\pi^2/|\mGraph|^2$.

To conclude the proof, let us observe that equality in \eqref{eq:bandlevy} clearly holds if $\mathcal{G}$ is a symmetric necklace, as it becomes clear considering the Rayleigh quotient of such a graph, which by homogeneity reduces to that of an individual interval of half length.

Conversely, let equality in \eqref{eq:bandlevy} hold for a given eigenfunction $\psi_2$; up to shortening $\Graph$ we can without loss of generality assume $\psi_2$ to vanish only on a Lebesgue zero set.
Then the support $\mathcal{G}_\pm$ of the positive/negative part of $\psi_2$ satisfies  $|\mathcal{G}_\pm|=|\mGraph|/2$. 
Now, because the increasing rearrangement $\psi_2^*$ of $\psi_2$ satisfies
\begin{displaymath}
	\frac{\int_{\mathcal{G}_+} |\psi_2'|^2\ud x}{\int_{\mathcal{G}_+} |\psi_2|^2 \ud x}
	= \frac{\int_{\mathcal{G}_+^\ast} |(\psi_2^\ast)'|^2\ud x}{\int_{\mathcal{G}_+^\ast}|\psi_2^\ast|^2\ud x},
\end{displaymath}
we deduce as above that $n(t)=\eta=2$ for a.e.\ $t\in [0,\max \psi]$.
Thus, up to a null set, $\mathcal{G}$ must
consist of two paths representing the pre-images of the set $(\min \psi_2,\max\psi_2)$. This is only possible if $\mathcal{G}$ is a (possibly degenerate) symmetric necklace.
Furthermore, any two parallel edges must have equal length: as we have already remarked in Lemma~\ref{LemmaSymmetrisation}, because we are assuming equality in~\eqref{eq:symmetrisation} we necessarily have $|\psi'_2(x)| = |\psi'_2(y)|$ for any $x$ and $y$ in the same level set $\psi_2^{-1}(t)$. We conclude that $\psi_2$ is identical along the two paths, hence the paths have the same length.
\end{proof}

\begin{proof}[Proof of Theorem~\ref{TheoremBKKM1}]
Since $\Graph$ has finite volume, $\mathcal H_{\Dir}$ has compact resolvent and there exists a positive ground state $\psi_1 \in H^1_0(\mGraph) \cap C(\overline \mGraph) \cap \oplus_{\me \in \mE} C^1([0, \ell_\me])$. 
In particular, the function $\psi_1$ can only vanish at $\Dir$, where it attains its minimum $m$.
Therefore, $\psi_1$ is an $H^1(\mGraph)$-function which takes all values in $[0,M]$ in $\overline \mGraph$, where $M := \max_{x \in \overline \mGraph} \psi_1(x)$.
Lemma~\ref{LemmaSymmetrisation} implies that the Rayleigh quotient of $\psi_1$ is at least $\essinf_{t \in [0,M]} n(t)^2$ times the Rayleigh quotient of its increasing rearrangement $\psi_1^*$, which in turn is no smaller than $\left( \frac{\pi}{2 |\mGraph|} \right)^2$, that is, the lowest eigenvalue of the Laplacian on $ [0,|\mGraph|]$ with mixed Dirichlet/Neumann boundary conditions. 

Since $\overline \mGraph$ is compact and $\psi_1(x) \to 0$ whenever $x \to \Dir$, every value in $(0,M]$ with $M := \max_{x \in \Graph} \psi_1 (x)$ will be attained at least once by $\psi_1$. 
We claim that every such value will in fact be attained at least twice, with the possible exception of a countable set. 
Indeed, if for some $t \in (0,M]$ there exists exactly one $x$ such that $\psi
_1(x) = t$, then cutting the graph at $x$ disconnects $\Graph$ into an upper and a lower level set of $\psi_1$. 
By our topological assumptions on $\Graph$, this is only possible if $x$ is a vertex, of which there are only countably many.
Thus $\essinf_{t \in [0,M]} n(t) \geq 2$, which concludes the proof of~\eqref{eq:bandlevy}.

To discuss the case of equality in~\eqref{eq:bkkm1}, observe that the symmetrisation process yields equality if and only if $n(t)= 2$ for a.e.\ $t$, which in turn implies that $\Graph$ is a symmetric necklace, with the minimum of $\psi_1$ attained at precisely one of its extremities -- either a vertex or an end, by assumption. Observe that in the former case, the necklace may still be either infinite (and in this case $\mathcal H_{\Dir}$ would be a mixed Friedrichs/Neumann realisation) or finite.
\end{proof}

To complete the proof of Theorem~\ref{FriedlanderTh1Gl1}, we next observe that, as in the compact case, it suffices to consider the case where $\mGraph$ is a (now possibly infinite)  tree; this follows from a basic surgery argument.

\begin{lemma}\label{baumlemma}
Let $ \mGraph $ be a locally finite metric graph of finite volume. If $\mGraph'$ is any graph obtained from $\mGraph$ by cutting through $ \mGraph $ countably many times, then
\begin{displaymath}
	\mu_k (\mGraph) \geq \mu_k(\mGraph')\qquad\hbox{for all }k\in\N.
\end{displaymath}
If $\mGraph'$ is obtained from $\mGraph$ by cutting through $j\in\mathbb N$ times, then
\begin{displaymath}
	\mu_k (\mGraph) \leq \mu_{k+j}(\mGraph')\qquad\hbox{for all }k\in\N.
\end{displaymath}
\end{lemma}

We refer to~\cite[Definition~3.2]{BerKenKur19} for the definition of cutting through a vertex; we note that the definition of \emph{cuts} originally given for compact metric graphs makes equal sense in the infinite case, since cutting through vertex is a local graph operation and the cutting procedure in Lemma~\ref{baumlemma} is iterative, i.e., we do not need to cut through infinitely many vertices simultaneously. The proof of Lemma \ref{baumlemma} is analogous to the one of \cite[Theorem~3.4]{BerKenKur19} and therefore omitted.

Note that it is always possible to obtain a tree from any locally finite graph $\mGraph$, upon cutting a countable number of times. 
We will thus assume for the rest of the section that $\mGraph$ is a (locally finite, connected) tree. In what follows, we will need the following notation. For an arbitrary point $ x \in \mGraph $ the graph $ \mGraph \backslash \{ x \} $  consists of $ p_x \in \N_0 $ connected components. We denote the closure of these $ p_x$  components by $ \left\{ \mGraph^{1}(x), \hdots \mGraph^{p_x}(x)\right\} $; these graphs are also trees. Inductively, we can remove a finite number of points $ \{x_1, \ldots x_n \} $. The set of the resulting graphs will be denoted by $ \mGraph(x_1, \ldots, x_n) $, and an element of this set with $ k(x_1, \ldots, x_n) $ elements by $ \mGraph^j(x_1, \ldots x_n) $ with $ 1 \leq j \leq k(x_1, \ldots, x_n) $.

The next lemma is a generalisation of \cite[Lemma 4]{Fri05}.

\begin{lemma}\label{fried4unendlich}
Let $ \mGraph $ be a locally finite metric tree with finite volume $|\mGraph| >0 $. For every $ 0 < l < |\mGraph| $  there exists some $ x \in \mGraph $ such that, for the subgraphs $ \left\{ \mGraph^{1}(x), \ldots, \mGraph^{p_x}(x)\right\}$ and an appropriate indexing, we have
\[ |\mGraph^{1}(x)| \leq |\mGraph|-l \quad \text{and} \quad  |\mGraph^{i}(x)| \leq l \qquad \text{for all} \quad  2 \leq i \leq p_x.  \]
\end{lemma}

\begin{proof} We take $ \varepsilon >0 $ with $ \min( |\mGraph| -l, l ) < |\mGraph| - \varepsilon $. We take a compact exhaustion $ (\mGraph_n)_{n\in\N} $ of $ \mGraph $ and fix $ n $ large enough that $ |\mGraph \backslash \mGraph_n| < \varepsilon $. Now consider $ \mGraph_n $. This graph is compact, and $\mGraph \setminus \mGraph_n$ consists of at most finitely many (possibly infinite) graphs $\widehat\mGraph_{k}$, $k=1,\ldots, m$, attached as pendants to $\mGraph_n$ at finitely many vertices (cf.\ the discussion after Proposition~\ref{PropBetti}). The sum of the volumes of these $ \widehat\mGraph_k $ does not exceed $ \varepsilon $. 
We replace each of these $ m $ graphs $ \widehat{\mGraph}_k $ by an edge $ \me_k $ of length $|\widehat{\mGraph}_k| $. The resulting graph, call it $ \widehat{\mGraph}_n $, satisfies the conditions of \cite[Lemma~4]{Fri05}. Thus, we can choose the point $ x $ from this lemma. This point $ x $ cannot be an element of the interior of an edge $ \me_k $, since otherwise one of the resulting graphs would have volume exceeding $ |\mGraph| - \varepsilon $. Since $ x $ cannot be a vertex of degree one, we conclude that $ x \in \mGraph_n $. Per construction $ x $ also has the desired properties in the graph $ \mGraph $.
\end{proof}

The following generalises \cite[Lemma~2]{Fri05}; the proof in the compact case may be repeated verbatim and is thus omitted.

\begin{lemma}\label{FriedL2}
Let $ \mGraph $ be a locally finite metric tree with finite volume $|\mGraph| >0 $ and let $ n \in \mathbb{N} $ with $ n \geq 2 $. Then there exist $ n-1 $ points $ x_1, \ldots x_{n-1} \in \mGraph $ such that, for any element of the set $ \mGraph(x_1, \ldots, x_{n-1}) $, we have
\[ \left|\mGraph^j(x_1, \ldots x_{n-1})\right| \leq \frac{|\mGraph|}{n} \qquad \hbox{for all }1 \leq j \leq k(x_1, \ldots, x_{n-1}).\]
\end{lemma}

We can finally give the proof of Theorem~\ref{FriedlanderTh1}, which closely follows that of~\cite[Theorem~1]{Fri05}.

\begin{proof}[Proof of Theorem~\ref{FriedlanderTh1}]
Let $\psi_j$ be any eigenfunction associated with $\mu_j$, $j \in \N$, and fix $k \geq 2$. Then for any $x_1,\ldots,x_{k-1} \in \mGraph$ there exists a nontrivial linear combination $\psi$ of the $k$ eigenfunctions $\psi_1,\ldots,\psi_k$ such that $\psi(x_i) = 0$ for all $i=1,\ldots,k$ (note that the existence of such a $\psi$ only requires the linear independence of the $\psi_j$; it does not require any properties of their nodal counts). We choose the points $x_i$ according to Lemma~\ref{fried4unendlich}, that is, in such a way that (keeping the notation of the previous lemmata) for every subgraph $ \mGraph^j(x_1, \ldots ,x_{k})  $, with $1 \leq j \leq k(x_1, \ldots, x_{k}) $  belonging to the set $ \mGraph(x_1, \ldots x_k) $,
\[ \left|\mGraph^j(x_1, \ldots , x_{k-1})\right| \leq \frac{|\mGraph|}{k} \qquad 1 \leq j \leq k(x_1, \ldots, x_{k}).\]
Now since $\psi$ is nontrivial there exists at least one subgraph, call it $\mGraph^1$, on which $\psi$ does not vanish identically. An argument similar to the one used to obtain \eqref{eq:nodaldomainproperty} shows that
\begin{equation}\label{friedsatz1}
	\int_{\mGraph^1 } |\psi'(x)|^2\, \ud x \leq \mu_k(\mGraph) \int_{\mGraph^1 } |\varphi(x)|^2\, \ud x.
\end{equation}
Since $\psi$ satisfies a Dirichlet condition at at least one point of $\mGraph^1$ (namely whichever of the $x_1,\ldots,x_n$ belong(s) to the boundary of $\partial\mGraph^1$), Lemma~\ref{FriedlanderTh1} implies
\begin{equation}\label{friedsatz2}
	\int_{\mGraph^1 } | \psi'(x) |^2\, \ud x \geq \frac{\pi^2}{4 |\mGraph^1|}^2 \int_{\mGraph^1} |\varphi(x)|^2\,\ud x.
\end{equation}
The fact that $|\mGraph^1| \leq |\mGraph|/k$ leads to \eqref{FriedlanderTh1Gl1}.

We still have to show the case of equality. The only difference to the proof for the finite case is that instead of segments with a Dirichlet endpoint $ y $ we may also have infinite path graphs of finite length, with a Dirichlet endpoint $ y $. Note that at the graph end of any such path the same Neumann condition holds as for a leaf of a finite path graph. Applying in particular Lemma~\ref{FriedL2}, the rest of the proof carries over verbatim from \cite[Section~2]{Fri05}, so we omit it.
\end{proof}

\subsection{Further lower bounds}
	\label{subsec:further_lower_bounds}

\subsubsection{Diameter}

Apart from the volume, an important quantity in spectral geometry is the diameter. It was proved in~\cite[Section~5]{KenKurMal16} that neither lower nor upper bounds on the spectral gap are generally possible in terms of diameter alone. 
However, estimates can be obtained if the diameter is complemented by volume, or if the graph has a special topological structure. 
In what follows we extend such estimates to infinite graphs.

We start with a bound in terms of volume and diameter. 
In the case of compact metric graphs, the following theorem is found in \cite[Theorem 4.4.6]{Plu22}; see also \cite[Theorem~1.1]{Ken20} and  \cite[Theorem~7.2]{KenKurMal16} for earlier iterations.

\begin{prop}\label{thm:kkmmLowerDiameter}
Let $ \mGraph $ be a locally finite, connected metric graph with finite volume $|\mGraph| >0 $ and diameter $\diam(\mGraph)> 0$, and finite Betti number. Then
\begin{equation}\label{eq:kkmmLowerDiameter}
	\mu_2 (\mGraph) \ge \frac{2}{|\mGraph|\diam(\mGraph)}.
\end{equation}
\end{prop}

\begin{rem}
	\label{rem:ladder}
Proposition~\ref{thm:kkmmLowerDiameter} holds on any $\Graph$, not necessarily of finite Betti number, on which one can find a compact exhaustion $(\mGraph_n)_n$ such that $\diam (\mGraph_n)$ converges to $(\diam \mGraph)$. One such example would be an infinite or semi-infinite ladder with suitably chosen edge lengths (cf.\ Example~\ref{exa:marvin-ladder}.(2)).
\end{rem}

\begin{proof}
Choose a compact exhaustion of $\mGraph$, and apply \cite[Theorem 4.4.6]{Plu22}. Proposition~\ref{PropDiameterConvergence} and Lemma~\ref{TheoremEigenvalueApproximation}(2) then yield \eqref{eq:kkmmLowerDiameter} on $\Graph$.
\end{proof}

One can, similarly, generalise the result of \cite[Theorem~1.2]{Ken20}, which gives a lower bound on $\mu_k$ in terms of $|\mGraph|$ and $\diam(\mGraph)$, to locally finite graphs with finite Betti number.
This estimate can be greatly improved if $\Graph$ is a (finite or infinite) tree. In this case, we obtain:

\begin{prop}\label{prop:estimate-tree-diameter}
Let $\Graph$ be a locally finite metric tree and let
	\[
	\Dir=\mathfrak{C}(\Graph)\cup\{\mv\in\mV : \deg(\mv)=1\}
	.
	\]
If $\Graph$ has finite diameter $\diam(\mGraph)>0$, then the lowest eigenvalue $\lambda_1(\Graph, \Dir)$ of $\mathcal H_\Dir$ satisfies
	\begin{equation}\label{eq:estimate-tree-diameter}
	\lambda_1(\Graph, \Dir)\ge \frac{\pi^2}{\diam(\mGraph)^2}.
	\end{equation}
\end{prop}

Note that due to Proposition~\ref{prop:length_epsilon_compact}, the assumptions of Proposition~\ref{prop:estimate-tree-diameter} imply that the spectrum of $\mathcal H_\Dir$ is purely discrete.
The proof of Proposition~\ref{prop:estimate-tree-diameter} follows precisely along the lines of \cite[Lemma~4.6]{BerKenKur17} and is therefore omitted.

\begin{rem}
Other lower bounds on $\lambda_1 (\Graph)$ are obtained in \cite[Corollary~3.6 and Theorem~4.1.(i)]{Sol04} in terms of both diameter and further related quantities (the \emph{height} and \emph{reduced height} of the tree), but under the additional assumption that the tree is \emph{radially symmetric} and a Dirichlet condition is imposed on its root.

The discussion following~\cite[Proposition~2.4]{Car00} shows that the assumption that $\Graph$ is a tree cannot generally be dropped.
\end{rem}

\subsubsection{Inradius}
	\label{subsubsec:Inradius}

Another natural geometric quantity is the \emph{inradius} 
\[
\Inr(\mGraph,\Dir):=\sup\{\dista_{\mGraph}(x,\Dir):x\in \mGraph \},
\]
that is, the supremum of radii of closed balls within $\overline \Graph$ that do not intersect the Dirichlet set $\Dir \subset \mathfrak{C}(\Graph)\cup \mV$.
In $\R^2$, the \emph{Makai inequality}~\cite{Mak65} (see also \cite{Hersch60}) states that there exists an absolute constant $C > 0$ such that for all bounded and simply connected domains $\Omega \subset \R^2$ the lowest eigenvalue $\lambda_1(\Omega)$ of the Dirichlet Laplacian on $\Omega$ satisfies
\[
\lambda_1(\Omega) \geq \frac{C}{\Inr(\Omega)^2},
\]
where $\Inr(\Omega)$ is the inradius of $\Omega$, 
\[
\Inr(\Omega) := \sup \{ r>0: \text{ there exists } x \in \Omega \text{ such that } B_r(x)\subset \Omega\}.
\]

It is fairly easy to see that a metric graph analogue of the Makai inequality can only hold on trees (as the Makai inequality itself holds only on simply connected domains) with a Dirichlet condition imposed on all ends and on all degree one vertices.
Let us next prove a Makai-type inequality for a special class of such trees, namely trees with a \emph{centre point} (or \emph{centre vertex}). 

\begin{defi}
	\label{def:centre_vertex}
	Suppose $\mGraph$ is a locally finite metric tree of finite inradius. Let $\centervertex$ be a vertex of $\mGraph$ and let
	\begin{equation}\label{eq:Dir-set-center-tree}
	\Dir:=\mathfrak{C}(\Graph)\cup\{\mv\in\mV : \deg(\mv)=1,~\mv\neq\centervertex\}
	.
	\end{equation}
	We say that $\centervertex$ is a \emph{centre vertex} of $\mGraph$ if $\dist (\centervertex, \gamma_1)=\dist(\centervertex,\gamma_2)$ holds for any $\gamma_1,\gamma_2\in \Dir$. 
	Whenever (at least) a centre vertex exists, we call $\mGraph$ a \emph{centred tree}.
\end{defi}

We refer to \autoref{fig:no_compact_embedding} for an example of a metric tree with \textit{two} centre vertices.
In particular, in the trees we consider, there is at most one vertex of degree one (namely $\centervertex$, if it is of degree one) where standard vertex conditions are imposed, whereas on all all other vertices, we impose Dirichlet conditions. Note that, if $\centervertex$ is a centre vertex of $\mGraph$, then
\begin{equation}
\label{eq:centre-inradius}
	\Inr(\mGraph,\Dir)=\dist(\centervertex,\gamma)\qquad\hbox{for all  }\gamma\in \Dir. 
\end{equation}
This identity motivates the use of centre vertices in this context.

\begin{theo}
\label{prop:Makai}
Suppose $\Graph$ is a locally finite metric tree of finite inradius with a centre vertex $\centervertex$ and let $\Dir$ be as in \eqref{eq:Dir-set-center-tree}. Then the lowest eigenvalue $\lambda_1(\Graph, \Dir)$ of $\mathcal H_\Dir$ admits the lower bound
	\begin{equation}
	\label{eq:Makai}
	\lambda_1(\Graph, \Dir)
	\geq 
	\frac{\pi^2}{4 \Inr(\Graph,\Dir)^2}.
	\end{equation}
	Equality holds if and only if $\mGraph$ is an equilateral star graph with Dirichlet conditions at all degree one vertices.
\end{theo}

\begin{rem}
\label{rem:two:centres}
We emphasise that the existence of a centre vertex is strictly weaker than radial symmetry as investigated in~\cite{Sol04}. 
This can be seen on the the diagonal comb graphs of Section~\ref{H10yesH1not}, which have two centres.
Example~\ref{exa:letter_T} below shows that the assumption on the existence of a centre vertex in Theorem~\ref{prop:Makai} cannot be dropped.
\end{rem}

The inequality~\eqref{eq:Makai} is a direct consequence of Proposition~\ref{prop:estimate-tree-diameter} since under the assumption of the existence of a centre vertex of degree greater than $1$, the diameter is twice the inradius; while for trees with a centre vertex of degree $1$ the inequality follows by mirroring the tree at the centre vertex. 
However, in order to classify the case of equality, we provide a different proof here using the following surgery principle, which we believe to be interesting in its own right.

\begin{lemma}\label{lem:cutting_Dirichlet}
	Let $\Graph$ be a locally finite, connected metric graph, and {$\Dir \subset \{\mv\in\mV~|~\deg(\mv)=1\} \cap \mathfrak{C} (\Graph)$.}
	Suppose that $\Graph_1$ and $\Graph_2$ are two closed subgraphs of $\Graph$ such that $(\Graph_1,\Graph_2)$ is a partition of $\Graph$, i.e., $\Graph=\Graph_1\cup\Graph_2$ holds, and $\Graph_1\cap\Graph_2$ consists of finitely many vertices of $\Graph$.
	Let $\psi$ denote a nonnegative and nontrivial eigenfunction corresponding to $\lambda_1(\Graph, \Dir)$ and let $\Dir_1:=\Graph_1\cap\Dir$.
	If, for each vertex $\mv\in \Graph_1\cap\Graph_2$ and each edge $\me$ connecting $\mv$ with a vertex $\mw$ in $\mGraph_2$, the derivative of $\psi$ { at $\mv$ pointing into $\me$}, is nonpositive then the inequality
	\begin{equation}\label{eq:surgery-dir}
	\lambda_1(\Graph, \Dir) 
	\geq
	\lambda_1(\Graph_1, \Dir_1)
	\end{equation}
	holds. The inequality is strict if and only if at least one of the above-mentioned derivatives is strictly negative.
\end{lemma}

\autoref{fig:demonstration-of-cutting-lemma} visually describes a possible case of application of Lemma~\ref{lem:cutting_Dirichlet}.

\begin{figure}[ht]
\begin{minipage}{0.49\textwidth}
\centering
\begin{tikzpicture}
\coordinate (a) at (-1,1);
\coordinate (b) at (-1,-1);
\coordinate (c) at (0,0);
\coordinate (d) at (1,1);
\coordinate (e) at (1,-1);
\coordinate (f) at (2,1);
\coordinate (g) at (3,1);
\coordinate (h) at (2,-1);
\draw[thick]  (a) edge (c);
\draw[thick]  (b) edge (c);
\draw[thick]  (d) edge (c);
\draw[thick]  (e) edge (c);
\draw[thick]  (e) edge (d);
\draw[thick]  (f) edge (d);
\draw[thick]  (f) edge (e);
\draw[thick]  (g) edge (f);
\draw[thick]  (h) edge (e);
\draw[fill=white] (a) circle (1.75pt);
\draw[fill] (b) circle (1.75pt);
\draw[fill] (c) circle (1.75pt);
\draw[fill] (d) circle (1.75pt);
\draw[fill] (f) circle (1.75pt);
\draw[fill] (e) circle (1.75pt);
\draw[fill=white] (h) circle (1.75pt);
\draw[fill=white] (g) circle (1.75pt);
\node at (.9,1) [anchor=south] {$\mv_1$};
\node at (.9,-1) [anchor=north] {$\mv_2$};
\node at (1.5,1) [anchor=south] {$\me_1$};
\node at (1.5,-1) [anchor=north] {$\me_3$};
\node at (1.6,0) [anchor=west] {$\me_2$};
\end{tikzpicture}
\end{minipage}
\begin{minipage}{0.49\textwidth}
\centering
\begin{tikzpicture}
\coordinate (a) at (-1,1);
\coordinate (b) at (-1,-1);
\coordinate (c) at (0,0);
\coordinate (d) at (1,1);
\coordinate (d1) at (1.4,1);
\coordinate (e) at (1,-1);
\coordinate (e1) at (1.4,-1);
\coordinate (f) at (2.4,1);
\coordinate (g) at (3.4,1);
\coordinate (h) at (2.4,-1);
\draw[thick]  (a) edge (c);
\draw[thick]  (b) edge (c);
\draw[thick]  (d) edge (c);
\draw[thick]  (e) edge (c);
\draw[thick]  (e) edge (d);
\draw[thick]  (f) edge (d1);
\draw[thick]  (f) edge (e1);
\draw[thick]  (g) edge (f);
\draw[thick]  (h) edge (e1);
\draw[thick,dashed]  (1.2,-0.7) edge (1.2,-1.3);
\draw[thick,dashed]  (1.2,0.7) edge (1.2,1.3);
\draw[fill=white] (a) circle (1.75pt);
\draw[fill] (b) circle (1.75pt);
\draw[fill] (c) circle (1.75pt);
\draw[fill] (d) circle (1.75pt);
\draw[fill] (f) circle (1.75pt);
\draw[fill] (e) circle (1.75pt);
\draw[fill] (d1) circle (1.75pt);
\draw[fill] (e1) circle (1.75pt);
\draw[fill=white] (h) circle (1.75pt);
\draw[fill=white] (g) circle (1.75pt);
\node at (-.6,0) {$\Graph_1$};
\node at (2.6,0) {$\Graph_2$};
\end{tikzpicture}
\end{minipage}
\caption{A partition of a graph $\Graph$ into two subgraphs $\Graph_1$ and $\Graph_2$. To apply Lemma \ref{lem:cutting_Dirichlet} the derivatives of the eigenfunction $\psi$ at the vertices $\mv_1$ and $\mv_2$ pointing into the respective edges $\me_1$, $\me_2$ and $\me_3$ have to be nonpositive.}\label{fig:demonstration-of-cutting-lemma}
\end{figure}

\begin{proof}
	Let $\mE_1$ denote the edge set of $\Graph_1$ and let $\psi_1$ denote the restriction of $\psi$ to $\Graph_1$. For a given vertex $\mv\in \Graph_1\cap\Graph_2$ let $\mE_{2,\mv}$ denote the set of edges in $\mGraph_2$ that are incident to $\mv$, and suppose that for each such edge $\me\simeq [0,\ell_\me]$ the vertex $\mv$ is identified with $0$. Then, our assumption states that $\psi'_\me(0)\leq 0$ for all $\me\in\mE_{2,\mv}$. A direct calculation using integration by parts yields
	\begin{align*}
		\lVert \psi_1' \rVert_{L^2(\Graph_1)}^2
		&=
		\sum_{\me \in \mE_1}
		\int_0^{\ell_\me} \lvert \psi'_\me(x) \rvert^2 \mathrm{d} x
		=
		\sum_{\me \in \mE_1}
		\psi_\me \psi'_\me \mid_0^{\ell_\me}
		-
		\sum_{\me \in \mE_1}
		\int_0^{\ell_\me} \psi_\me(x) \psi''_\me(x) \mathrm{d} x
		\\
		&=
		\sum_{\mv\in\mGraph_1\cap\mGraph_2}\psi (\mv)\sum_{\me\in\mE_{\mv,2}} \psi'_\me(0)
		+
		\lambda_1(\Graph, \Dir) \lVert \psi_1 \rVert_{L^2(\Graph_1)}^2\\
		& \leq \lambda_1(\Graph, \Dir) \lVert \psi \rVert_{L^2(\Graph_1)}^2,
	\end{align*}  
	whence
	\begin{equation}
	\label{eq:surgery_inradius_pre_minmax}
	\frac{\lVert \psi_1' \rVert_{L^2(\Graph_1)}^2}{\lVert \psi_1 \rVert_{L^2(\Graph_1)}^2} \leq \lambda_1(\Graph, \Dir).
	\end{equation}
	In light of the variational principle \eqref{eq:varchar-lambda1-mix}, this implies \eqref{eq:surgery-dir}. 
	In order to classify the case when \eqref{eq:surgery-dir} is a strict inequality, we first note that~\eqref{eq:surgery_inradius_pre_minmax} is strict if and only if $\psi'_{\me}(0) < 0$ for some $\mv\in\Graph_1\cap\Graph_2$ and some $\me\in\mE_{2,\mv}$. 
	Therefore, also \eqref{eq:surgery-dir} is strict in this case.
	Conversely, if $\psi'_{\me}(0) = 0$ for all $\mv\in\Graph_1\cap\Graph_2$ and all $\me\in\mE_{2,\mv}$, then the restricted function $\psi_1 = \psi|_{\Graph_1}$ satisfies Kirchhoff conditions on $\mGraph_1$ in the vertices in $\mGraph_1\cap\mGraph_2$ and is therefore a bona fide eigenfunction of $\mathcal H_{\Dir_1}$, and since $\psi_1$ is strictly positive except at the Dirichlet vertices, it must be associated with the first eigenvalue $\lambda_1(\mGraph,\Dir)$ of $\mathcal{H}_{\Dir_1}$.
This implies equality in \eqref{eq:surgery-dir}.
\end{proof}
Note that we will only use Lemma~\ref{lem:cutting_Dirichlet} in its simplest form, where $\mGraph_1\cap\mGraph_2$ consists of exactly one vertex $\mv$ of $\mGraph$ and where there is exactly one edge $\me$ of $\mGraph$ that connects $\mv$ with a vertex $\mw\in\mGraph_2$. In that case, one only needs to check the sign of one derivative of $\psi$ to apply Lemma~\ref{lem:cutting_Dirichlet}.
\begin{proof}[Proof of Theorem~\ref{prop:Makai}]
We prove~\eqref{eq:Makai} in three steps.

\textit{Step 1:} Suppose first that $\Graph$ is a compact tree and $\deg(\centervertex)=1$.
 If $|\Dir|=1$, then $\Graph$ is isometrically isomorphic to an interval with mixed Dirichlet/Neumann conditions in the degree one vertices and therefore $\lambda_1(\mGraph,\Dir)=\frac{\pi^2}{4 \Inr(\Graph,\Dir)^2}$. Next assume that $|\Dir|\geq 2$. Let $\psi$ denote a nonnegative eigenfunction corresponding to $\lambda_1(\Graph,\Dir)$. Using induction and the Kirchhoff condition, it can be shown that there exists a path $\mathcal P$ in $\mGraph$ connecting the centre point $\centervertex$ and a vertex $\mv\in\Dir$ such that $\psi$ is decreasing along $\mathcal P$. Since $|\Dir|\geq 2$, and hence $\Graph$ is not an interval, $\mathcal P$ passes at least one vertex other than $\mv, \centervertex$. Let $\mw$ denote the unique vertex in $\mV\setminus \Dir$ that is adjacent to $\mv$, let $\mGraph'$ denote the graph obtained after removing the edge $\mv\mw$ from $\mGraph$ and let $\Dir':=\Dir\setminus\{\mv\}$. Since $\psi$ is decreasing on the edge $\mv\mw$, we may apply Lemma \ref{lem:cutting_Dirichlet} to obtain $\lambda_1(\mGraph,\Dir)\geq \lambda_1(\mGraph', \Dir')$; however, by \eqref{eq:centre-inradius}, $\Inr (\mGraph, \Dir) = \Inr (\mGraph', \Dir')$. Repeating this argument inductively, after a finite number of steps we are reduced to $|\Dir| = 1$, which proves the estimate \eqref{eq:Makai} for all compact trees whose centre vertex has degree $1$.

\textit{Step 2:} Suppose now that $\Graph$ is a compact tree and $\centervertex$ has degree $d>1$. Again let $\psi$ denote a nonnegative eigenfunction corresponding to $\lambda_1(\Graph,\Dir)$. Since $\psi$ satisfies Kirchhoff conditions in $\centervertex$, there exists an edge $\me$ incident to $\centervertex$ such that $\psi$ has nonpositive derivative on $\me$ at $\centervertex$. Now, cutting through the centre vertex $d-1$ times yields $d$ disjoint trees; for each of these, $\centervertex$ continues to be the centre vertex (and in particular each has the same inradius). From these trees, let $\Graph_2$ denote the tree containing the edge $\me$ and let $\mGraph_1$ denote its complement in $\Graph$, the (restored) union of the other $d-1$ trees. In $\Graph_1$ the centre vertex $\centervertex$ has degree $d-1$. Setting $\Dir_1:=\Dir\cap\Graph_1$ and applying Lemma \ref{lem:cutting_Dirichlet} we obtain $\lambda_1(\mGraph,\Dir)\geq \lambda_1(\mGraph_1,\Dir_1)$. An induction argument together with Step 1 now yields \eqref{eq:Makai} for all compact trees.

\textit{Step 3:} Finally, we suppose that $\Graph$ is an arbitrary (infinite) tree satisfying the assumptions of Proposition \ref{prop:Makai}. We consider the compact exhaustion $(\mGraph_n)_{n\in\mathbb N}$ of $\mGraph$ with 
	\[\mGraph_n:=\left\{x\in\mGraph : \dista_\mGraph(x,\centervertex)\leq \left(1-\frac{1}{2n}\right)\Inr(\mGraph,\Dir)\right\},\quad n\in\mathbb N.\]
For each $n\in\mathbb N$, the graph $\mGraph_n$ is a (compact) tree graph with centre vertex $\centervertex$ and inradius
	\[\Inr(\mGraph_n,\partial \mGraph_n)=\left(1-\frac{1}{2n}\right)\Inr(\mGraph,\Dir)\]
by \eqref{eq:centre-inradius}. Therefore, using \eqref{eq:Makai} in the compact case, we find
	\[\lambda_1(\mGraph_n,\partial \mGraph_n)\geq \frac{\pi^2}{4\Inr(\mGraph_n,\partial \mGraph_n)^2}\qquad\hbox{for all }n\in \N.\]
Letting $n \to \infty$ and using Lemma~\ref{TheoremEigenvalueApproximation}(1), we obtain \eqref{eq:Makai} in the infinite case.

\emph{Identifying the case of equality}. 
It it straightforward to see that equality in~\eqref{eq:Makai} holds if $\mGraph$ is an equilateral star graph.
Conversely, assume that $\mGraph$ is as in the statement of the theorem but \emph{not} an equilateral star graph.
Then, there must be a path $\mathcal{P}$ within $\mGraph$, leading from the centre vertex $\centervertex$ to $\Dir$ on which $\psi_1$ is monotonously decreasing.
If the subtree $\tilde \mGraph$, rooted at $\centervertex$, which contains $\mathcal{P}$ was isomorphic to a path graph itself, we can remove it and by Lemma~\ref{lem:cutting_Dirichlet} this will not increase $\lambda_1(\mGraph, \Dir)$.
Since $\centervertex$ has finite degree, we can only perform this operation a finite number of times and can assume without loss that $\mGraph$ is a tree where every path leading from $\centervertex$ to $\Dir$ will encounter at least one vertex of degree larger than three.
Now again, there is a path $\mathcal{P}$, leading from $\centervertex$ to $\Dir$ on which $\psi_1$ is monotonously decreasing and strictly monotone away from $\centervertex$.
This path will encounter at least one vertex $\mv$ of degree at least three and the derivative of $\psi_1$ at $\mv$ in the direction of $\mathcal{P}$, will be strictly decreasing.
Consequently, we can cut the subtree, attached at $\mv$, which contains the rest of $\mathcal{P}$ and by Lemma~\ref{lem:cutting_Dirichlet}, obtain a new graph $\mGraph' \subset \mGraph$ with Dirichlet conditions at $\Dir' \subset \Dir$, satisfying the conditions of Theorem~\ref{prop:Makai}, but with $\lambda_1(\mGraph', \Dir') < \lambda_1(\mGraph, \Dir)$.
Since inequality~\eqref{eq:Makai} holds for both $\lambda_1(\mGraph, \Dir)$ and $\lambda_1(\mGraph', \Dir')$, we see that equality cannot hold for $\lambda_1(\mGraph, \Dir)$.
\end{proof}

	We now show that without the existence of a centre point, \eqref{eq:Makai} need not hold. 
\begin{exa}
	\label{exa:letter_T}
	Take $\Graph_T$ to be the 3-star graph depicted in Figure~\ref{fig:letter_T}, consisting of two edges of length $\ell_1$ and one edge of length $\ell_2 < \ell_1$, that meet at a single vertex $\mv$; in this case $\Inr (\Graph_T, \Dir) = \frac{\ell_1+\ell_2}{2}$.
	Take $\Dir$ to be the set of the three degree one vertices. 
	We claim that 
	\[
	\lambda_1(\mGraph_T, \Dir) < \frac{\pi^2}{4 \Inr (\Graph_T, \Dir)^2}.
	\]
	Indeed, call $\Graph_T^+$ the graph obtained by adding another edge of length $\ell_2$ to $\mv$ and call $\Dir^+$ the set of its degree one vertices.
	The graph $\Graph_T^+$ can be understood as two Dirichlet intervals $[0, \ell_1 + \ell_2]$, glued at one point.
	By symmetry, the restriction of any ground state $\psi$ (i.e. eigenfunction for $\lambda_1 (\mGraph_T,\Dir)$) must coincide with the corresponding Dirichlet ground states on each interval.
	On the one hand, this implies 
	\[
	\lambda_1(\mGraph_T^+, \Dir^+) 
	= 
	\frac{\pi^2}{(\ell_1 + \ell_2)^2}
	=
	\frac{\pi^2}{4 \Inr(\Graph, \Dir)^2},	
	\]
	on the other hand, the outgoing derivatives of the (nonnegative) ground state at $\mv$ on the two shorter edges are strictly decreasing.
	Consequently, by Lemma~\ref{lem:cutting_Dirichlet}, removing one of those edges will strictly decrease $\lambda_1$, which is the claim.
\end{exa}

\begin{figure}[ht]
	\begin{tikzpicture}
		\begin{scope}	
			\draw (0,.4) node {$\Graph_T$};
			
			\draw[thick] (-2,0) -- (0,0);
			\draw[thick] (2,0) -- (0,0);
			\draw[thick] (0,-.5) -- (0,0);
		
			\draw[thick, fill = black](0,0) circle (3pt);
			\draw[thick, fill = white](-2,0) circle (3pt);
			\draw[thick, fill = white](2,0) circle (3pt);
			\draw[thick, fill = white](0,-.5) circle (3pt);
		\end{scope}
		
		\begin{scope}[xshift = 5cm]
			
			\draw (0,.4) node {$\Graph_T^+$};
			
			\draw[thick] (-2,0) -- (0,0);
			\draw[thick] (2,0) -- (0,0);
			\draw[thick] (-.2,-.5) -- (0,0);
			\draw[thick] (.2,-.5) -- (0,0);
		
			\draw[thick, fill = black](0,0) circle (3pt);
			\draw[thick, fill = white](-2,0) circle (3pt);
			\draw[thick, fill = white](2,0) circle (3pt);
			\draw[thick, fill = white](-.2,-.5) circle (3pt);
			\draw[thick, fill = white](.2,-.5) circle (3pt);
		\end{scope}

	\end{tikzpicture}

\caption{The metric graphs $\Graph_T$ and $\Graph_T^+$ from Example~\ref{exa:letter_T}}
	\label{fig:letter_T}
\end{figure}

\begin{rem}	
	Note that Example~\ref{exa:letter_T} does not exclude \emph{per se} the validity of a Makai inequality for the Friedrichs realisation on metric trees, but it shows that it cannot hold with constant $\pi^2/4$. 
Let us also emphasise that on two-dimensional domains, the optimal constant in the Makai inequality is unknown, Makai himself having proved it with $C = 1$.
However, Hersch~\cite{Hersch60} proved before Makai that the optimal constant on \emph{convex} two-dimensional domains is $\pi^2/4$; we may possibly regard Theorem~\ref{prop:Makai}  as an analogue of Hersch's result.

\end{rem}

\section{Upper bounds on the eigenvalues}\label{sec:other}

We finish with two upper bounds, which complement some of the lower bounds of Section~\ref{sec:symmetrisation}.

\begin{prop}\label{thm:upper-diam}
Let $ \mGraph $ be a locally finite, connected metric graph with finite volume $|\mGraph| >0 $ and diameter $\diam(\mGraph)>0$ and finite Betti number $\beta \in \N_0$.
Then $ \mGraph $ satisfies
\begin{equation}\label{eq:kkmmUpperDiameter}
	\mu_2 (\mGraph) \leq \frac{\pi^2}{\diam(\mGraph)^2}\frac{4|\mGraph| - 3\diam(\mGraph)}{\diam(\mGraph)}.
\end{equation}	
\end{prop}

Analogously to Remark~\ref{rem:ladder}, the conclusion of the theorem holds whenever one can find a compact exhaustion with diameter converging to $\diam(\mGraph)$.

\begin{proof}
The proof follows directly from combining \cite[Theorem~7.1]{KenKurMal16} (which gives \eqref{eq:kkmmUpperDiameter} on any compact graph) applied to any compact exhaustion of $\mGraph$, together with Proposition~\ref{PropDiameterBelow} and Lemma~\ref{TheoremEigenvalueApproximation}(3).
\end{proof}

We finish with the estimate mentioned above which is also new for compact graphs.

\begin{theo}\label{TheoremBettiD}
Let $\mGraph$ be a locally finite, connected metric graph with finite diameter $\diam(\mGraph)>0$, finite Betti number $\beta \in \N_0$, and for which the Neumann realisation $\mathcal{H}_{\mathrm{N}}$ has compact resolvent.
Then for all $k\geq 2$ we have
\begin{equation}\label{eq:BettiD}
	\mu_k (\mGraph) \leq  (k+\beta-1)^2\frac{\pi^2}{\diam(\mGraph)^2}.
\end{equation}
\end{theo}

In the case of a compact tree graph (i.e., for which $\beta = 0$) we recover a theorem of Rohleder \cite[Theorem~3.4]{Roh17}; for general compact graphs it improves the upper bound on $\mu_2$ in~\cite[Remark~6.3]{KenKurMal16}, $\mu_2 (\mGraph) \leq \frac{4\pi^2 |\mathcal E|^2}{\diam(\mGraph)^2}$, and provides a corresponding bound on the higher eigenvalues $\mu_k$, $k\geq 3$, for the first time.

\begin{proof}
Fix $k\geq 2$ and let $(\mGraph_n)_{n\in\N}$ be a compact exhaustion of $\mGraph$ for which $\diam(\mGraph_n) \to \diam(\mGraph)$ (see Proposition~\ref{PropDiameterConvergence}); we suppose $n$ to be large enough that all cycles of $\mGraph$ are contained in $\mGraph_n$, so that $\beta(\mGraph_n) = \beta$ (see Proposition~\ref{PropBetti}). We will prove \eqref{eq:BettiD} for $\mGraph_n$; the statement for $\mGraph$ then follows from Lemma~\ref{TheoremEigenvalueApproximation}(3).

Since $\mGraph_n$ has Betti number $\beta$, it is possible to cut it $\beta$ times to produce a tree $\mathcal{T}$, which obviously satisfies $\diam(\mathcal{T}) \geq \diam(\mGraph_n)$. 
By Lemma \ref{baumlemma}, we have $\mu_k (\mGraph_n) \leq \mu_{k+\beta} (\mathcal{T})$. 
Now \cite[Theorem~3.4]{Roh17} implies that
\begin{displaymath}
	\mu_k (\mGraph_n) \leq \mu_{k+\beta} (\mathcal{T}) \leq \frac{\pi^2 (k+\beta-1)^2}{\diam(\mathcal{T})^2} \leq \frac{\pi^2 (k+\beta-1)^2}{\diam(\mGraph_n)^2},
\end{displaymath}
which completes the proof.
\end{proof}

Observe that, for $\beta=0$, \eqref{eq:BettiD} is sharp for all $k$ (simply take $\Graph$ to be an interval), but it is not clear what happens for higher $\beta$. 
Simple examples such as a loop graph suggest that it might be rougher.

Another geometric quantity used in upper bounds on eigenvalues on compact metric graphs is the \emph{girth}, see \cite{BerKenKur23}. Let us precisely define it in the case of infinite metric graphs.

\begin{defi}
  \label{def:dirichlet-girth}
  The \emph{girth}  of a connected metric graph $\Graph$ shall be given by
  \begin{displaymath}
\Girth(\Graph):=   \inf \{|\mathfrak{c}| \colon \mathfrak{c} \subset \Graph \text{ is
      a cycle in $\Graph'$} \},
  \end{displaymath}
  where $\Graph'$ is the metric graph obtained from $\Graph$ by
  identifying (or ``gluing together'') all Dirichlet vertices of
  $\Graph$, if any are present; and all ends, if any are present.  The girth is defined to be zero if   $\Graph'$ is a tree.
\end{defi}

If finite length of $\mGraph$ is assumed, then for girth to be a nontrivial quantity there can only be finitely many cycles (and ends) in $\mGraph$, since otherwise the infimum will necessarily be zero. In this case, we can immediately reproduce the bound \cite[Proposition 1.8]{BerKenKur23}.

\begin{prop}
Let $ \mGraph $ be a locally finite, connected metric graph with finite volume $|\mGraph| >0 $. Then
\[
	\mu_2(\mGraph)
	<
	\frac{48 L}{\Girth(\mGraph)^3}.
\]
\end{prop}

The proof requires adapting to infinite graphs the construction of a suitable test function as done for finite graphs in the proof of \cite[Proposition 1.8]{BerKenKur23}: we omit the easy details.

Other bounds, such as \cite[Theorem 1.2]{BerKenKur23}, stating that in the presence of at least one Dirichlet vertex, 
\[
\lambda_1(\mGraph, \mathfrak{V}) \leq \frac{\pi^2}{\Girth(\mGraph)^2}
\]
while also plausible on infinite metric graphs, will require more technical effort to prove on infinite graphs since it relies on a surgery principle and $\delta$-couplings, which we do not discuss in this article.

\appendix

\section{Relation of the Freudenthal compactification to the metric completion}

\label{app:completions}

Throughout this section let \(\mGraph\) be a locally finite metric graph. Our aim is to discuss the relation between the completion \(\overline{\mGraph}\) of \(\mGraph\) as a metric space and its Freudenthal compactification \(\mGraph\cup \mathfrak{C}(\Graph)\), where \(\mathfrak{C}(\Graph)\) denotes the set of topological ends of \(\mGraph\) (see Definition~\ref{Topologisches Ende}). The following proposition shows that each point in the \emph{metric boundary} \(\overline{\mGraph}\setminus\mGraph\) can be associated with a unique end in \(\mathfrak{C}(\Graph)\).
\begin{prop}\label{prop:identifiying-metric-with-top-ends}
	For each \(x\in\overline{\mGraph}\setminus\mGraph\) and each sequence \((x_n)_{n\in\mathbb N}\) in \(\mGraph\) converging to \(x\) there exist a topological end \(\gamma\in \mathfrak{C}(\Graph)\) and a sequence \((U_n)_{n\in\mathbb N}\) representing \(\gamma\), such that \(x_n\in U_n\) for all \(n\in\mathbb N\).
\end{prop}
\begin{proof}
	Fix some \(x\in\overline{\mGraph}\setminus\mGraph\) and a Cauchy sequence \((x_n)_{n\in\mathbb N}\) in \(\mGraph\) converging to \(x\). We prove the claimed existence of \(\gamma\) and \((U_n)_{n\in\mathbb N}\). Consider a compact exhaustion \((\Graph_k)_{k\in\mathbb N}\) of \(\Graph\). Since \(\Graph_k\) is compact in \(\Graph\), we have \(\dista_\Graph(x,\Graph_k)>0\) for all \(k\in\mathbb N\). Therefore, for each \(k\) there exists some \(n_k\in\mathbb N\) with
\begin{equation}\label{eq:cauchy-property-distance}
	\dista_\Graph(x_n,x_m)<\frac{\dista_\Graph(x,\Graph_k)}{2}
\end{equation}
for all \(n,m\geq n_k\), and without loss of generality, we may assume \(n_k\leq n_{k+1}\) for all \(k\). It is then immediate that
\begin{equation}\label{eq:distance-xn-graphk}
	\dista_\Graph(x_n,\Graph_k)\geq \frac{\dista_\Graph(x,\Graph_k)}{2},
\end{equation}
and, thus, \(x_n\notin \Graph_k\) hold for all \(k\in\mathbb N\) and all \(n\geq n_k\). We now claim that, for fixed \(k\), all \(x_n\) with \(n\geq n_k\) are in the same connected component of \(\Graph\setminus \Graph_k\). Indeed, consider arbitrary \(m,n\geq n_k\) and an arbitrary, rectifiable path \(c:[0,1]\rightarrow\Graph\) connecting \(c(0)=x_n\) and \(c(1)=x_m\). If \(x_n\) and \(x_m\) were in two different components of \(\Graph\setminus \Graph_k\), then we could find some \(t\in (0,1)\) with \(y:=c(t)\in\Graph_k\). Using \eqref{eq:distance-xn-graphk}, and denoting by $L(c)$ the length of the (image of the) path $c$ in $\Graph$, we would then obtain
\begin{align*}
	L(c) 
    = L(c_{|[0,t]}) + L(c_{|[t,1]})
	& 
    \geq \dista_\Graph(x_n,y)+ \dista_\Graph(y,x_m) \\
	& 
    \geq \dista_\Graph(x_n,\Graph_k)+ \dista_\Graph(\Graph_k,x_m) 
    \geq \dista_\mGraph(x,\Graph_k).
\end{align*}
Since \(c\) was arbitrary, this would imply 	\(\dista_\mGraph(x_n,\Graph_k)\geq \dista_\mGraph(x,\Graph_k)\) in contradiction to \eqref{eq:cauchy-property-distance}.

Now, for \(k\in\mathbb N\), let \(V_k\) denote the connected component of \(\Graph\setminus \Graph_k\) with \(x_n\in V_k\) for all \(n\geq n_k\). Then, \(\Graph_k\subset \Graph_{k+1}\) yields \(V_k\supset V_{k+1}\) for all \(k\). Moreover, from \(\bigcup_{k\in\mathbb N} \mathrm{int}(\Graph_k)=\Graph\), we obtain \(\bigcap_{k\in\mathbb N} \overline{V_k} =\emptyset\). Furthermore, using that \(\Graph\) is locally connected, it can be shown that \(\partial V_k\subset \Graph_k\) and, thus, \(\partial V_k\) is compact. Therefore, \((V_k)_{k\in\mathbb N}\) is a representative of a topological end \(\gamma\). We proceed by constructing a representative \((U_n)_{n\in\mathbb N}\) of \(\gamma\) with \(x_n\in U_n\) for all \(n\in\mathbb N\). For \(n\in\mathbb N\), we define
	\[U_n:=\begin{cases}
		\Graph, & \text{if }n<n_1,\\
		V_k, & \text{if }n_k\leq n <n_{k+1};
	\end{cases}\]
it is easily checked that \((U_n)_{n\in\mathbb N}\) is indeed a representative of \(\gamma\) with \(x_n\in U_n\) for all \(n\in\mathbb N\). This completes the proof of Proposition \ref{prop:identifiying-metric-with-top-ends}.
\end{proof}
	The following lemma shows that the topological end \(\gamma\) in Proposition \ref{prop:identifiying-metric-with-top-ends} does not depend on the specific choice of Cauchy sequence converging to \(x\).
\begin{lemma}
	Suppose \((x_n)_{n\in\mathbb N}\) and \((x_n')_{n\in\mathbb N}\) are Cauchy sequences in \(\mGraph\) converging to the same \(x\in\overline{\mGraph}\). Suppose also that \((U_n)_{n\in\mathbb N}\) and \((U_n')_{n\in\mathbb N}\) are representatives of respective graph ends \(\gamma\) and \(\gamma'\) such that \(x_n\in U_n\) and \(x_n'\in U_n\) for all \(n\in\mathbb N\). Then \(\gamma = \gamma'\).
\end{lemma}
\begin{proof}
	We first prove \(\overline{U_n}\cap \overline{U_n'}\neq\emptyset\) for all \(n\in\mathbb N\). Assume for a contradiction that \(\overline{U_{n_0}}\cap \overline{U_{n_0}'}\) is empty for some \(n_0\in\mathbb N\). Now, for \(n\geq n_0\),  consider an arbitrary path \(c:[0,1]\rightarrow \mGraph\) connecting \(x_n=c(0)\) and \(x_n'=c(1)\). Because \(\overline{U_{n_0}}\) and \(\overline{U_{n_0}'}\) are disjoint, there exist \(0\leq t<t'\leq 1\) with \(c(t)\in\partial U_{n_0}\) and \(c(t')\in\partial U_{n_0}'\). We obtain
	\begin{align*}
		L(c) & = L(c_{|[0,t]}) + L(c_{|[t,t']}) + L(c_{|[t',1]})\\
		& \geq \dista_\mGraph(c(t),c(t'))\\
		& \geq \dista_\mGraph(\partial{U_{n_0}},\partial{U_{n_0}'}).
	\end{align*}
	Since \(c\) and \(n\geq n_0\) are arbitrary, we obtain	
	\begin{equation}\label{eq:dist-between-cauchy-seq}
	\dista_\mGraph(x_n,x_n')\geq \dista_\mGraph(\partial{U_{n_0}},\partial{U_{n_0}'})
	\end{equation}
	for all \(n\geq n_0\). Since \(\partial U_{n_0}\) and \(\partial U_{n_0}'\) are compact and disjoint, we have \(\dista_\mGraph(\partial{U_{n_0}},\partial{U_{n_0}'})>0\). Thus, \eqref{eq:dist-between-cauchy-seq} is a contradiction to \((x_n)_{n\in\mathbb N}\) and \((x_n')_{n\in\mathbb N}\) having the same limit in \(\overline{\mGraph}\). In conclusion, \(U_n\cap U_n'\) must be non-empty for all \(n\in\mathbb N\).
	
	Next, we prove \(\gamma=\gamma'\). Let \(n\in\mathbb N\) be fixed. Since \(\partial U_n\) is compact and the open sets \((\mathcal G \setminus \overline{U_j'})_{j\in\mathbb N}\) exhaust \(\mathcal G\), there exists some \(j\geq n\) with \(\partial U_n\subset \mathcal G\setminus U_j'\). This implies that
	\begin{equation}\label{eq:for-connect-argument}
		U_j'\subset U_n\cup \mathcal G\setminus \overline{U_n}
	\end{equation}
holds. Note that \(U_j'\cap U_n\) is a superset of the non-empty set \(U_j'\cap U_j\) and is thus non-empty as well. Therefore, since both sets on the right-hand side of \eqref{eq:for-connect-argument} are open and \(U_j'\) is connected, we obtain \(U_j'\cap (\Graph\setminus \overline{U_n})=\emptyset\) and, thus, \(U_j'\subset U_n\). Analogously, one shows the existence of \(k\in\mathbb N\) with \(U_k\subset U_j'\). Therefore, \((U_n)_{n\in\mathbb N}\) and \((U'_n)_{n\in\mathbb N}\) are equivalent with respect to the relation from Definition \ref{Topologisches Ende}, which in turn yields \(\gamma = \gamma'\).
\end{proof}
\begin{rem}
	Proposition \ref{prop:identifiying-metric-with-top-ends} can be extended to all connected and locally connected length spaces \(X\) that have a compact exhaustion \((K_n)_{n\in\mathbb N}\) with \(K_n\subset \mathrm{int}(K_n)\) for all \(n\).
\end{rem}
\begin{defi}
	We define the map \(\boundarymap:\overline{\Graph}\rightarrow \Graph\cup\mathfrak{C}(\Graph)\) as follows:
	\begin{enumerate}
	\item for \(x\in\Graph\), we set \(\boundarymap(x):=x\),
	\item for \(x\in \overline{\Graph}\setminus\Graph\), let \(\boundarymap(x)\in \mathfrak{C}(\Graph)\) be the unique topological end from Proposition \ref{prop:identifiying-metric-with-top-ends}.
	\end{enumerate}
\end{defi}
The following lemma provides different characterisations of \(\boundarymap\):
\begin{lemma}\label{lem:char-boundary-map}
	Let \(x\in \overline{\Graph}\setminus\Graph\) and \(\gamma\in\mathfrak{C}(\mGraph)\). Moreover, let \((U_n)_{n\in\mathbb N}\) be a representative of \(\gamma\). Then, the following statements are equivalent:
	\begin{enumerate}[(i)]
	\item \(\boundarymap(x)=\gamma\),
	\item for each Cauchy sequence \((x_k)_{k\in\mathbb N}\) in \(\Graph\) converging to \(x\) and each \(n\in\mathbb N\), there exists some \(k_n\in\mathbb N\) with \(x_k\in U_n\) for all \(k\geq k_n\),
	\item for each \(n\in\mathbb N\), there exists some \(\varepsilon>0\), such that \(\{y\in\Graph~|~\dista_\mGraph(x,y)<\varepsilon\}\subset U_n\).
	\end{enumerate}
\end{lemma}
\begin{proof}
	We first prove the implication (i) $\implies$ (ii). By definition of \(\boundarymap\), \(\boundarymap(x)=\gamma\) yields that there is a representative \((V_k)_{k\in\mathbb N}\) of \(\gamma\) with \(x_k\in V_k\) for all \(k\in\mathbb N\). By equivalence, for \(n\in\mathbb N\), there exists \(k_n\in\mathbb N\) with \(V_{k_n}\subset U_n\). For \(k\geq k_n\) we obtain \(x_k\in V_k\subset V_{k_n}\subset U_n\). This yields (ii).
	
	For the implication (ii) $\implies$ (iii), suppose (iii) does not hold. Then, for some \(n\in\mathbb N\), we find a sequence \((x_k)_{k\in\mathbb N}\) with \(x_k\ni U_n\) and \(\dista_\mGraph(x,x_k)<\frac{1}{k}\) for all \(k\in\mathbb N\). Therefore, \((x_k)_{k\in \mathbb N}\) is a Cauchy sequence converging to \(x\) that does not satisfy (ii). By contraposition (ii) implies (iii).
	
	It remains to prove that (iii) implies (i). By (iii), for \(n\in\mathbb N\), there exists some \(\varepsilon_n>0\) with
		\[\{y\in\Graph~|~\dista_\mGraph(x,y)<\varepsilon_n\}\subset U_n.\]
	We pick some \(x_n\in\mGraph\) with \(\dista_\mGraph(x,y) < \min(\varepsilon_n,\frac{1}{n})\). Then, by construction, \((x_n)_{n\in\mathbb N}\) is a Cauchy sequence in \(\mGraph\) converging to \(x\) with \(x_n\in U_n\) for all \(n\in\mathbb N\). By definition of \(\boundarymap\), we obtain \(\boundarymap(x)=\gamma\).
\end{proof}
\begin{rem}
	We note that Lemma \ref{lem:char-boundary-map}(iii) implies that \(\boundarymap\) is in fact a continuous map when \(\overline{\Graph}\) is equipped with the topology of the completion of \(\Graph\) as a metric space and \(\Graph\cup\mathfrak{C}(\Graph)\) is equipped with the topology of the Freudenthal compactification of \(\Graph\).
\end{rem}
It is well known that, for locally finite graphs, the topological ends of a metric graph may be uniquely identified with its so-called graph ends (cf.~\cite{DieKue03}), which yields that its topological ends are purely dependent on the combinatorial structure of the graph. This is no longer true for its metric ends, as will be demonstrated by the following example:
\begin{exa}\label{exa:marvin-ladder}
We consider the infinite ladder graph \(\Graph\) as depicted in Figure \ref{fig:ladder-graph}. The ladder graph has exactly one topological end \(\gamma\). However, hereinafter, we consider four different choices for the lengths of the graph edges \(\me_n, \mf_n^1\) and \(\mf_n^2\) for \(n\in\mathbb N\) (see Figure \ref{fig:ladder-graph} for the notation), each of which leads to a different structure of the metric boundary of \(\Graph\).

\textit{(1)} If all edges of \(\Graph\) have the same length \(1\), \(\Graph\) is complete and, therefore, its metric boundary is empty. In particular, \(\boundarymap\) is not surjective.

\textit{(2)} If the edge lengths of \(\Graph\) are given by  \(\ell_{\me_n}=\ell_{\mf_n^1}=\ell_{\mf_n^2}=\frac{1}{2^n}\) for all \(n\in\mathbb N\), then \(\Graph\) has finite length and, in particular, \(\gamma\) has finite length. Furthermore, it can be shown the metric boundary of \(\Graph\) consists of exactly one point.

\textit{(3)} Suppose next that the edge lengths of \(\Graph\) are given by \(\ell_{\mf_n^1}=\ell_{\mf_n^2}=\frac{1}{2^n}\) and \(\ell_{\me_n}=1\) for all \(n\in\mathbb N\). Then, \(\gamma\) has infinite length, and \(\overline{\Graph}\setminus\Graph\) consists of two points.  In particular, \(\boundarymap\) is not injective.

\textit{(4)} Finally, we consider the case where the edge lengths of \(\Graph\) are \(\ell_{\mf_n^1}=\frac{1}{2^n}\) and \(\ell_{\me_n}=\ell_{\mf_n^2}=1\) for \(n\in\mathbb N\).  Then, \(\gamma\) is a graph end of infinite length and \(\overline{\Graph}\setminus\Graph\) consists of precisely one point.

We finish this example by comparing the cases (2) and (4), where \(\boundarymap\) is bijective. We remark first that, in (2), it can be shown \(\boundarymap\) is in fact a homeomorphism  and, thus, the topologies of \(\overline \Graph\) and \(\Graph\cup \mathfrak C(\Graph)\) coincide. However, in (4), \(\boundarymap\) is not a homeomorphism. Indeed, the set
	\[U:= \{x\}\cup\bigcup_{n\in\mathbb N} (\mf_n^1\cup\me_1)\]
is a neighbourhood of the unique point \(x\in \overline \Graph\setminus\Graph\), whereas its image
	\[\boundarymap(U)= \{\gamma\}\cup\bigcup_{n\in\mathbb N} (\mf_n^1\cup\me_1)\]
is not a neighbourhood of \(\gamma =\boundarymap(x)\).
\end{exa}
The rest of this section will be dedicated to study the phenomena observed in Example \ref{exa:marvin-ladder} in full generality. We begin with the following:
\begin{prop}\label{prop:unique-ends-finite}
	Suppose that \(\gamma \in \mathfrak{C}_0(\mGraph)\) is a topological end of \(\Graph\) with finite volume. Then there exists a unique \(x\in\overline \Graph\setminus\Graph\) with \(\boundarymap(x)=\gamma\).
\end{prop}
\begin{proof}
	We first prove the existence of \(x\). Let \((U_n)_{n\in\mathbb N}\) be a representative of \(\gamma\). For each \(n\in\mathbb N\), we pick some \(x_n\in U_n\); then, for \(n\geq m\), we have \(x_m\in U_n\). Since each \(U_n\) is connected, we obtain
		\[\dista_\mGraph(x_n,x_m)\leq |U_n|\]
	for all \(n\geq m\). Because \(\gamma\) has finite length, we have \(\lim_{n\rightarrow \infty} |U_n|=0\). With the previous inequality, we obtain that \((x_n)_{n\in\mathbb N}\) is a Cauchy sequence. Because
		\[\bigcap_{n\in\mathbb N} \overline{U_n} =\emptyset,\]
	the limit \(x\) of \((x_n)_{n\in\mathbb N}\) must be in \(\overline{\mGraph}\setminus\mGraph\). Moreover, by definition of \(\boundarymap\) we have \(\boundarymap(x)=\gamma\).
	
	It remains to show the uniqueness of \(x\). Let \(x,y\in \overline{\mGraph}\setminus\mGraph\) with \(\boundarymap(x)=\boundarymap(y)=\gamma\). Let \((U_n)_{n\in\mathbb N}\) be a representative of \(\gamma\) and let \((x_n)_{n\in\mathbb N}\) and \((y_n)_{n\in\mathbb N}\) be two Cauchy sequences in \(\mGraph\) converging to \(x\) and \(y\) respectively. Let \(n\in\mathbb N\). By Lemma \ref{lem:char-boundary-map}(ii), there exists some \(k_n\in\mathbb N\) with \(x_{k_n},y_{k_n}\in U_n\). We obtain
		\[\dista_\mGraph(x_{k_n},y_{k_n})\leq |U_n|.\]
	Then, \(\lim_{n\rightarrow \infty} |U_n|=0\) yields \(\dista_\mGraph(x,y)=0\) and, thus, \(x=y\).	
\end{proof}
Motivated by (3) and (4) in Example \ref{exa:marvin-ladder} we introduce the following classification of graph ends:
\begin{defi}
	A graph end of infinite volume \(\gamma\) of \(\mGraph\) with \(\boundarymap^{-1}(\gamma)\neq \emptyset\) is called a \emph{degenerate graph end}.
\end{defi}
For the infinite ladder graph, we have already seen in Example \ref{exa:marvin-ladder} that the topologies of the metric completion and the Freudenthal compactification do not coincide if there the graph has degenerate graph ends. However, if a general metric graph does not have any degenerate graph ends, the two topologies do in fact coincide, as shown by the following
\begin{prop}\label{prop:boundary-map-open}
	If \(\Graph\) has no degenerate graph ends, \(\boundarymap\) is open.
\end{prop}
\begin{proof}
Consider an open set \(W\subset \overline{\mGraph}\). We have to show that \(\boundarymap(W)\) is open with respect to the topology of the Freudenthal compactification. It suffices to show that each \(\gamma\in \boundarymap(W)\cap \mathfrak{C}(\mGraph)\) is an interior point of \(\boundarymap(U)\). By definition of Freudenthal compactification we need to show that, for a representative \((U_n)_{n\in\mathbb N}\) of \(\gamma\), there exists some \(m\in\mathbb N\) with \(U_m\subset \boundarymap(W)\).

Let \(x\in W\) with \(\boundarymap(x)=\gamma\). Let \(\varepsilon>0\) with
	\[U_\varepsilon(x):= \{y\in\Graph~|~\dista_\mGraph(x,y)<\varepsilon\} \subset W.\]
	We claim that there exists some \(n\in\mathbb N\) with \(U_n\subset U_\varepsilon(x)\). Assume for a contradiction that for each \(n\in\mathbb N\), there exists some \(y_n\in U_n\) with \(\dista_\mGraph(x,y_n)\geq \varepsilon\). As in the proof of Proposition \ref{prop:unique-ends-finite}, it follows that \((y_n)_{n\in\mathbb N}\) converges to some \(y\in \overline{\mGraph}\setminus \mGraph\) with \(\boundarymap(y)=\gamma\). Moreover, by construction of \((y_n)_{n\in\mathbb N}\) we have \(\dista_\mGraph(x,y)\geq \varepsilon\) and, thus, \(x\neq y\). However, by assumption, \(\gamma\) must be a finite graph end. Therefore, \(x\neq y\) contradicts Proposition \ref{prop:unique-ends-finite}.
	
	We conclude that there exists some \(n\in\mathbb N\) with \(U_n\subset U_\varepsilon(x)\subset \boundarymap(W)\). This completes the proof.
\end{proof}
As an immediate consequence of Proposition \ref{prop:unique-ends-finite}, Proposition \ref{prop:boundary-map-open} and the fact that the Freudenthal compactification is open, we obtain the following
\begin{cor}\label{cor:all_graphs_end_finite_volume}
	If all graph ends of \(\Graph\) have finite volume, \(\boundarymap\) is a homeomorphism. In particular, \(\overline{\mGraph}\) is compact.
\end{cor}
\bibliographystyle{plain}
\bibliography{literatur}
\end{document}